\documentclass[a4paper,11pt,twoside]{article}
\usepackage[utf8]{inputenc}
\usepackage[T1]{fontenc}
\usepackage[english]{babel}
\usepackage{amssymb}
\usepackage[dvipsnames]{xcolor}
\usepackage{amsmath}
\usepackage{amsfonts}
\usepackage{dsfont}
\usepackage{amsthm}
\usepackage{fancyhdr}
\usepackage{esint}
\usepackage{authblk}
\usepackage{multicol}
\usepackage[cm]{fullpage}
\usepackage{mathabx}
\usepackage{bbold}
\usepackage{graphicx}
\usepackage{hyperref}
\usepackage{colortbl}
\usepackage{algorithm}
\usepackage{algpseudocode}
\usepackage{placeins}

\usepackage{enumitem}

\newcommand{\q}{q_T}
\newcommand{\Q}{Q_T}
\newcommand{\Sig}{\Sigma_T}

\usepackage{hyperref}
\newtheorem{thm}{Theorem}[section]
\newtheorem{coro}{Corollary}[section]
\newtheorem{ppt}{Proposition}[section]

\newtheorem{lem}{Lemma}[section]
\newtheorem{remark}{Remark}[section]
\numberwithin{equation}{section}

\author{Nicolae Cîndea}
\author[1]{Geoffrey Lacour \thanks{Corresponding author: Geoffrey Lacour - \texttt{geoffrey.lacour@uca.fr}}}
\affil[1]{Université Clermont Auvergne, CNRS, LMBP, F-63000 Clermont-Ferrand, France}
\title{Null controllability of quasilinear parabolic equations\\
  with gradient dependent coefficients}
\date{}

\begin{document}

\maketitle

 
 
\begin{abstract} 
\noindent
The aim of this paper is to study the null controllability of a class of quasilinear parabolic equations. In a first step we prove that the associated linear parabolic equations with non-constant diffusion coefficients are approximately null controllable by the means of regular controls and that these controls depend continuously to the diffusion coefficient. A fixed-point strategy is employed in order to prove the null approximate controllability for the considered quasilinear parabolic equations. We also show the exact null controllability in arbitrary small time for a class of parabolic equations including the parabolic $p$-Laplacian with $\frac{3}{2} < p < 2$. The theoretical results are numerically illustrated combining a fixed point algorithm and a reformulation of the controllability problem for linear parabolic equation as a mixed-formulation which is numerically solved using a finite elements method.
\end{abstract}

\section{Introduction}
\label{sec:introduction}

This paper consider the approximate null controllability of quasilinear equations of the following form:
\begin{equation}\label{eq:0}
\left\{
\begin{array}{ll}
\partial_t y - {\rm div}\left(F(\lvert \nabla y \rvert)\nabla y\right) = \chi_{\omega}\varphi& \text{ in } \Q\\
y = 0& \text{ on } \Sig\\
y(0)= y^0 & \text{ in } \Omega,\\
\end{array}
\right.
\end{equation}

where $\Omega \subset \mathbb{R}^N$ is an open bounded domain with a smooth boundary $\partial \Omega$, $\Q = (0, T) \times \Omega$, $\Sig = (0 , T) \times \partial \Omega$, the initial data $y^0$ belongs to $L^2(\Omega)$, and there exists $p > 1$ such that the function $F:\mathbb{R}_+ \to \mathbb{R}_+^*$ verifies the following assumptions:

\begin{enumerate}[label=(A\arabic*)]
    \item \label{A1} $F \in W^{1,\infty}(\mathbb{R}_+) \cap C^{\infty}(\mathbb{R}_+) \cap L^2(\mathbb{R}_+) \cap L^{\frac{p}{p-1}}(\mathbb{R}_+)$;
    \item \label{A2} The potential defined for every $t \in \mathbb{R}_+$ by $\Phi(t) = \int_0^t sF(s)\; ds$ is convex and satisfies $\Phi \in W^{1,\frac{p}{p-1}}\left(\mathbb{R}_+\right)$;
    \item \label{A3} There exists $C_1, C_2, \mu, \nu > 0$ and $k_1,k_2 \geq 0$ such that for every $t \in \mathbb{R}_+$, we have that 
    $$k_1 + C_1(\mu + t^2)^{\frac{p-2}{2}} \leq F(t) \leq k_2 + C_2(\nu + t^2)^{\frac{p-2}{2}}.$$
\end{enumerate}
The control $\varphi$ acts in the open and non empty set $\omega \subset \Omega$.  More precisely, we denote by $\chi_{\omega} \in C^\infty(\overline\Omega)$ a regular function such that

\[
\chi_{\omega}(x) = 
\left\{
\begin{array}{cl}
1 & \text{ for } x \in \omega_\delta\\
0 & \text{ for } x \in \overline{\Omega \setminus \omega},
\end{array}
\right.
\]
where $\omega_\delta = \{x \in \omega \text{ such that } \text{ dist}(x, \partial \omega) > \delta \}$ for a given $\delta > 0$ small enough. We also denote \(\q = (0, T) \times \omega\).

Let us underline that assumption \ref{A2} means that $\tilde{\Phi}(u) = \Phi(\lvert u \rvert)$ is a convex potential and thus the operator $A : X \to X^*$ defined for every $\phi \in X$ by
    \begin{equation}\label{eq:opA}
    A \phi = -\mathrm{div}\left(F(\lvert\nabla \phi \rvert)\nabla \phi \right)
    \end{equation}
    is monotone (see \cite[Chapitre 2, Définition 1.2.]{lions-quelques-resolutions}), where $X$ is a reflexive Banach space compactly and densely embedded in $L^2(\Omega)$ and $X^*$ is its dual with respect to the pivot space $L^2(\Omega)$. Having this in mind, we point out that the existence of a unique weak solution to equation~\eqref{eq:0} is a direct consequence of \cite[Chapitre 2, Section 1, Théorème 1.2.- bis]{lions-quelques-resolutions} applied to the nonlinear monotone operator $A$ given by~\eqref{eq:opA}, when $F$ is chosen such that it leads to the monotonicity of the operator $A$. From a historic point of view, the analysis of quasilinear parabolic equations and the properties of their solutions took off in the 1960's, with the pioneering works \cite{ladyzhenskaya-solonnikov-uraltceva,lions-quelques-resolutions,aronson-serrin-67}.  Quasilinear equations as \eqref{eq:0} are to be compared to the regularized parabolic $p$-Laplacian from which it derived their study. Such equations have been the source of a large number of publications in the last decades. As in the case of linear parabolic equations, the question of the boundedness of the solution or its gradient arises naturally. Such results are now  well known and have been established for large classes of quasilinear systems, we can refer to  \cite{chen-nakao-00,duzaar-mingione-09,aronson-serrin-67,cano-casanova-lopez-gomez-takimoto-12,byun-palagachev-shin-16,diening-scharle-schwarzacher-19,boccardo-orsina-porzio-21,porzio-21} in the case of a bounded domain with a regular boundary. However, this remains a source of an important research activity, especially for the study of singular or degenerate systems. Examples include recent second-order regularity results for the parabolic $p$-Laplacian (see e.g. \cite{CianchiMazya-2020,FengParviainenSarsa-2022}) and its alternative in the symmetrized gradient framework with the A-approximation method (see \cite{BerselliRuuzicka-2022}). In the case of \eqref{eq:0} whose nonlinearity satisfies the assumptions \ref{A1}-\ref{A3}, which give rise to a nondegenerate, nonsingular quasilinear equation, it is well known that smooth solutions are obtained (we refer e.g. to \cite[Theorem 3.4.1., Sections 3.1.4. and 0.10.]{CherrierMilani-2012}). Note that this fact is used, for example, to approximate the solution of the parabolic $p$-Laplacian equation (see e.g. \cite[Section 4]{FengParviainenSarsa-2022} or \cite{Lewis-1983}). In our study, we voluntarily set aside the $p=1$ case, which is more difficult, but for which it is still possible to show interesting existence and regularity properties (see e.g. \cite{nakao-ohara-96,Wiegner-1984,EnglerKawohlLuckhaus-1990}); in this last case, we underline that smoothness of viscosity solutions have been proved in \cite{KawohlKutev-1995}, this can be linked to reasoning such as that presented in \cite{JuutinenLindqvistManfredi-2001}. We should also point out that the smoothing effect of quasilinear parabolic equations is known in many cases (see e.g. \cite[Chapter IV]{Haraux-1981} or \cite{nakao-ohara-96}). The literature on the subject of quasilinear parabolic equations and systems is extremely vast, so we we mainly refer the interested reader to monographs \cite{lions-quelques-resolutions,ladyzhenskaya-solonnikov-uraltceva,lieberman,wang-21-nonlinear,PruessSimonett-2016,Zheng-2004,CherrierMilani-2012,Koshelev-1995,Amann-1995} for the study of the properties of these equations and systems.\\

The controllability of quasilinear equations has been recently studied, as in \cite{casas-fernandez-93,casas-fernandez-yong-95,casas-chrysafinos-18}, in the framework of optimal control, or \cite{liu-zhang-09,liu-12,fernandez-cara-limaco-marin-gayte-21} in the framework of exact controllability. In these last papers, the results for exact controllability hold for systems where the nonlinear term depends on the solution of the system, but not on its gradient. The local controllability of quasilinear equations with a gradient dependent term has been studied in the recent paper \cite{Fernandez-CaraLimacoThamstenMenezes-2023}. At our knowledge, the global controllability of such equations remains open and is the main purpose of the present work.

With the objective of applying a fixed point method in order to control the quasilinear equation, we first investigate the existence of smooth distributed controls for the following linear heat equation in divergence form with a space and time dependent diffusion coefficient:

\begin{equation}\label{eq:1}
\left\{
\begin{array}{ll}
\partial_tu - {\rm div}\left(a(t,x)\nabla u\right) = \chi_{\omega}\varphi& \text{ in } \Q\\
u = 0& \text{ on } \Sig \\
u(0) = y^0& \text{ in } \Omega\\
\end{array}
\right.
\end{equation}
in which we consider diffusion coefficients $a \in C^{\infty}(\overline{\Q})$ satisfying: 
\begin{equation}\label{eq:condition_diffusion-coeff}
0 < \rho_{\star} \leq a(t,x) \qquad ((t,x) \in \overline{\Q}),
\end{equation}
where $\rho_{\star} >0$ is a constant.

 The distributed controllability of equation~\eqref{eq:1} is a well studied subject (see for example~\cite{fursikov-imanuvilov} or the more recent review paper~\cite{sgefc}). The existence of an optimal distributed control, in the sense of minimal $L^2$-norm, can be obtained by applying the Hilbert Uniqueness Method (HUM) introduced in \cite{Lions-88}. The main idea of the method is to consider the dual final boundary value problem of \eqref{eq:1} given by:
\begin{equation}\label{eq:dual}
\left\{
\begin{array}{ll}
\partial_t\varphi + {\rm div}\left(a(t,x)\nabla \varphi\right) = 0 & \text{ in } \Q \\
\varphi = 0& \text{ on } \Sig\\
{\varphi}(T)= \varphi^0& \text{ in } \Omega\\
\end{array}
\right.
\end{equation}
for some $\varphi^0 \in L^2(\Omega)$ and then to deduce the (exact or approximate) controllability of \eqref{eq:1}. For fixing the notation, we denote $\mathcal{S}^a(\varphi^0) := \varphi$ the solution of~\eqref{eq:dual} associated to the final data $\varphi^0$. Remark that using this notation we enhance the dependence of solutions of~\eqref{eq:dual} on the diffusion coefficient $a$. The existence of approximate controls for linear parabolic equations by the use of this approach, even regular, is now well known (see \cite{lions-glowinski-he,boyer-13,Bensoussan-1993}). However, in order to apply a fixed-point theorem, it is necessary to establish the continuity of the controls with respect to the diffusion coefficient. 

\begin{thm}\label{thm:smooth-controls-linear-heat}
Let $\Omega$ be an open bounded subset of $\mathbb{R}^N$ with Lipschitz boundary, $y^0 \in L^2(\Omega)$, $a \in C^{\infty}(\overline \Q)$ satisfying~\eqref{eq:condition_diffusion-coeff} and $T > 0$. Then, for every $\varepsilon > 0$ there exists an approximate control $\varphi \in  C^{\infty}(\overline{\Q})$ in the sense that the corresponding solution $u$ of~\eqref{eq:1} verifies
\begin{equation}
\lVert u(T) \rVert < \varepsilon.
\end{equation}

Moreover, the control $\varphi$ depends Lipschitz continuously to the diffusion coefficient $a$  for the norm $\lVert \cdot \rVert_{L^2(\Q)}$.
\end{thm}

Here and henceforth we denote by $(\cdot,\ \cdot)$ the inner product in $L^2(\Omega)$ and by $\|\cdot\|$ the associated norm. The main result of the paper is a consequence of Theorem~\ref{thm:smooth-controls-linear-heat} and provides the approximate null controllability of quasilinear equation~\eqref{eq:0}.

\begin{thm}\label{thm:global-approximate-control-quasi}
  Assume that $F$ satisfies assumption \ref{A1}--\ref{A3} and $y^0$ belongs to $L^2(\Omega)$ are chosen such that there exists a unique solution of \eqref{eq:0}. Then, there exists a distributed control $\varphi$, whose regularity is given by Theorem~\ref{thm:smooth-controls-linear-heat}, such that \eqref{eq:0} is approximately null controllable in any time $T > 0$, i.e., for every $y^0 \in L^2(\Omega)$ and every $\varepsilon > 0$ there exists a control $\varphi \in C^{\infty}(\overline{\Q})$ such that the solution $y$ of~\eqref{eq:0} satisfies
  \[
    \|y(T)\| \le \varepsilon.
  \]
\end{thm}

In particular, Theorem~\ref{thm:global-approximate-control-quasi} implies the approximate null controllability of the so-called parabolic $p$-Laplacian:

\begin{equation}\label{eq:p-laplacian}
\left\{
\begin{array}{ll}
\partial_t\mathsf{v} - {\rm div}\left(\lvert \nabla \mathsf{v} \rvert^{p-2}\nabla \mathsf{v}\right) = \chi_{\omega}\varphi& \text{ in } \Q\\
\mathsf{v} = 0& \text{ on } \Sig\\
\mathsf{v}(0) = y^0& \text{ in } \Omega.\\
\end{array}
\right.
\end{equation}
with $\frac{3}{2} < p < 3$.

More exactly, we prove the following corollary. 

\begin{coro}\label{coro:approx-control-p-laplacian}
  Let $y^0 \in L^2(\Omega)$ and $\frac{3}{2} < p < 3$. Then~\eqref{eq:p-laplacian} is approximate null controllable in any time $T > 0$, i.e., for every $\varepsilon > 0$ there exists a control $\varphi \in C^{\infty}(\overline{\Q})$ such that the solution $\mathsf{v}$ of~\eqref{eq:p-laplacian} verifies
  \[
    \|\mathsf{v}(T)\| \le \varepsilon.
  \]
\end{coro}

In fact, in the case where the solution stops in finite time, and where this stopping time is well controlled by the norm of the initial data, we can show the global exact controllability of \eqref{eq:0}. More precisely, the following result holds.

\begin{thm}\label{thm:global-exact-control-quasi}
Assume that $F$ satisfies assumptions \ref{A1}--\ref{A3}, and that $y$ is the solution of \eqref{eq:0} associated to an initial data $y^0 \in L^2(\Omega)$. Moreover, let us consider that $y$ stops in finite time, which is that, if $\varphi=0$ then, there exists $T_s \in (0,T)$, $\gamma > 0$ and $\mu > 0$ such that:
\begin{equation}\label{eq:est-stopping-time}
    \lVert y(T_s) \rVert = 0 \quad \text{ and } \quad T_s \leq \mu \lVert y^0 \rVert^{\gamma}.
\end{equation}
Then, one can choose the force term $\varphi$ such that $y$ is exactly null controllable in any time $T^\star \in (0, T)$.
\end{thm}

Applying the results in \cite[Proposition
2.1.]{dibenedetto-degenerate-parabolic}, the following corollary is a
direct consequence of Theorem~\ref{thm:global-exact-control-quasi} and \cite[Exemple 1.5.2 and Théorème 1.2 bis]{lions-quelques-resolutions}, setting $X = W_0^{1,p}(\Omega) \cap L^2(\Omega)$.

\begin{coro}\label{coro:global-exact-control-p-laplacian}
  Let $y^0 \in L^\infty(\Omega)$ and $\frac{3}{2} < p < 2$. Then, by always choosing a non-negative solution to~\eqref{eq:p-laplacian}, the problem~\eqref{eq:p-laplacian} is null controllable in any time $T > 0$, \emph{i.e.}, there exists a control $\varphi \in C^{\infty}(\overline{\Q})$ such that the solution $\mathsf{v}$ of~\eqref{eq:p-laplacian} verifies
  \[
    \mathsf{v}(T) = 0.
  \]
\end{coro}

Also, another example is given by the following equation:

\begin{equation}\label{eq:p-laplacian-delta}
\left\{
\begin{array}{ll}
\partial_t\mathsf{u} - \Delta \mathsf{u}- {\rm div}\left(\lvert \nabla \mathsf{u} \rvert^{p-2}\nabla \mathsf{u}\right) = \chi_{\omega}\varphi& \text{ in } \Q\\
\mathsf{u} = 0& \text{ on } \Sig\\
\mathsf{u}(0) = y^0& \text{ in } \Omega\\
\end{array}
\right.
\end{equation}
with $\frac{3}{2} < p < 2$. We point out that the operator $A : \mathsf{u} \mapsto -\Delta \mathsf{u}- {\rm div}\left(\lvert \nabla \mathsf{u} \rvert^{p-2}\nabla \mathsf{u}\right)$ is well-defined and monotone over $X := H^2(\Omega) \cap H_0^1(\Omega)$ (see e.g. \cite[Section 4.3.]{barbu}). 

Then, we have the following result.

\begin{coro}\label{coro:global-exact-control-p-laplacian-delta}
  Let $y^0 \in L^2(\Omega)$, and $\frac{3}{2} < p < 2$. Then~\eqref{eq:p-laplacian-delta} is null controllable in any time $T > 0$, i.e., there exists a control $\varphi \in C^{\infty}(\overline{\Q})$ such that the solution $\mathsf{u}$ of~\eqref{eq:p-laplacian-delta} verifies
  \[
    \mathsf{u}(T) = 0.
  \]
\end{coro}

For the sake of clarity, we will omit throughout the article the dependence of the constants and will generically denote positive constants by $C$. 

The remaining part of this paper is structured as follows. Section~\ref{sec:control-linear} is dedicated to the existence of smooth Lipschitz continuous in $L^2$ approximate null control for the linear equation~\eqref{eq:1}. In order to prove the Theorem~\ref{thm:global-approximate-control-quasi} we employ a fixed point strategy described in Section~\ref{sec-control-quasilinear}. Finally, Section~\ref{sec:num} numerically illustrate the computation of controls in both the linear and non-linear frameworks.

\section{Approximate controllability of the linear equation}\label{sec:control-linear}

Let $\varepsilon > 0$. For every $\delta > 0$ we denote $M_{\delta}(\varphi^0) \in C^{\infty}(\Omega)$ a mollification of some $\varphi^0 \in L^2(\Omega)$ (see \cite{brezis}) such that $\|M_{\delta}(\varphi^0) - \varphi^0\| \to 0$ when $\delta \to 0$. Following this notation we set $M_0(\varphi^0) := \varphi^0$. For every $\delta \ge 0$ we consider the functional:

\begin{equation}\label{eq:reg-HUM-functional}
J^a_\delta(\varphi^0) = \frac{1}{2}\iint_{\q}\chi_{\omega}\lvert \mathcal{S}^a(M_\delta(\varphi^0)) \rvert^2\;dx\, dt + \frac{\varepsilon}{2}\lVert \varphi^0 \rVert^2 + (\mathcal{S}^a(M_\delta(\varphi^0))(0),y^0),
\end{equation}
where $\mathcal{S}^a(\varphi^0)$ is the solution of \eqref{eq:dual} with a diffusion coefficient $a$ satisfying 
\eqref{eq:condition_diffusion-coeff}. Let us point out that the standard HUM functional is nothing else than $J_0^a$ given by~\eqref{eq:reg-HUM-functional}.
Following the classical arguments as presented in \cite[Sections 1.2. and 1.3.]{boyer-13} or in \cite[Chapter 1]{lions-glowinski-he}, minimizers of $J^a_0$ can give rise to approximate controls $\varphi \in C^{\infty}([0, T) \times \Omega)$ of \eqref{eq:1}, \emph{i.e.}, the corresponding solution of~\eqref{eq:1} verifies
\begin{equation}
  \label{eq:approx-con-lin}
  \|u(T)\| < \varepsilon.
\end{equation}
More exactly, we have the following observability inequality which is the key ingredient of the proof of the null approximate controllability of~\eqref{eq:1}.

Let us recall the following result.
\begin{ppt}[Observability inequality {\rm \cite[Theorem 1.5.]{sgefc}}]\label{ppt:observability}
 There exists a constant $C_0 > 0$, depending of $\Omega$, $T$, $\omega$, and $\lVert a \rVert_{L^{\infty}(\Q)}$, such that the following inequality holds:

\begin{equation}\label{eq:observability}
\lVert \mathcal{S}^a(M_\delta(\varphi^0))(0) \rVert^2 \leq C_0 \iint_{\q} \chi_\omega \lvert \mathcal{S}^a(M_\delta(\varphi^0)) \rvert^2\; dx\, dt \qquad (\varphi^0 \in L^2(\Omega)).
\end{equation}
\end{ppt}

We aim to show here that minimizing $J^a_\delta$ over $L^2(\Omega)$ for $\delta > 0$ we obtain a control $\varphi_\delta$ which is now in $C^{\infty}(\overline{\Q})$ which is close to the one obtained by minimizing $J^a_0$, so, for $\delta>0$ small enough and for this regular control, the solution $u$ of~\eqref{eq:1} still verifies~\eqref{eq:approx-con-lin}. For every $\delta > 0$ we denote $\overline{\varphi^0_\delta}$ the minimum of $J^a_\delta$. Therefore, for every $\psi^0 \in L^2(\Omega)$ we have
\begin{align}
    \label{eq:optim-delta}
    \iint_{\q} \chi_\omega \mathcal{S}^a(M_\delta(\overline{\varphi_\delta^0})) \mathcal{S}^a(M_\delta(\psi^0)) \ dx \ dt + \varepsilon (\overline{\varphi_\delta^0}, \psi^0) + (\mathcal{S}^a(M_\delta(\psi^0))(0), y^0) = 0 \\
    \label{eq:optim}
    \iint_{\q} \chi_\omega \mathcal{S}^a(\overline{\varphi^0}) \mathcal{S}^a(\psi^0)  \ dx \ dt + \varepsilon (\overline{\varphi^0}, \psi^0) + (\mathcal{S}^a(\psi^0)(0), y^0) = 0.
\end{align}

\begin{lem}\label{lem:weak-limit-delta}
  With the above notation, there exists a constant $C_0 > 0$ such that
  \begin{equation}
    \label{eq:bound-varphi^0}
    \|\overline{\varphi_{\delta}^0}\| \le C_0 \|y^0\|
  \end{equation}
  and for every $f \in C([0, T], L^2(\Omega))$ the following convergences occur
\[
(f, \mathcal{S}^a(M_\delta(\overline{\varphi^0_\delta})) - \mathcal{S}^a(\overline{\varphi_\delta^0}))_{L^2(\Q)} \to 0,
\]
when $\delta \to 0$
\end{lem}

\begin{proof}
  Remark that $J_{\delta}^a(\overline{\varphi_\delta^0}) \leq J_\delta^a(0) = 0$. This implies that $(\overline{\varphi^0_\delta})_{\delta > 0}$  verifies~\eqref{eq:bound-varphi^0}. Then we can extract a subsequence, still denoted $(\overline{\varphi^0_\delta})_{\delta > 0}$, weakly converging to $\varphi^0$ in $L^2(\Omega)$.
  Let us observe that $\mathcal{S}^a(M_{\delta}(\overline{\varphi_{\delta}^0}) - \varphi_{\delta}^0) = \mathcal{S}^a(M_{\delta}(\overline{\varphi_\delta^0}) - M_{\delta}(\varphi^0)) + \mathcal{S}^a(M_{\delta}(\varphi^0) - \overline{\varphi_\delta^0})$.
 
 Now, in the weak formulation of $\mathcal{S}^a(M_{\delta}(\overline{\varphi_\delta^0}) - M_{\delta}(\varphi^0))$, one can write for the final datum term:
 \begin{equation}\label{eq:weak-limit-final-datum}
     (\mathcal{S}^a(M_{\delta}(\overline{\varphi_\delta^0}) - M_{\delta}(\varphi^0))(T),f(T))_{L^2(\Omega)} = (\overline{\varphi_\delta^0},M_\delta(f(T)))_{L^2(\Omega)} - (M_{\delta}(\varphi^0),f(T))_{L^2(\Omega)} \underset{\delta \to 0}{\longrightarrow} 0.
 \end{equation}
 
 Hence, $\underset{\delta \to 0}{\mathrm{lim}}\; \mathcal{S}^a(M_{\delta}(\overline{\varphi_\delta^0}) - M_{\delta}(\varphi^0))$ is the weak solution of \eqref{eq:dual} associated to the null final datum, namely arguing by uniqueness $\mathcal{S}^a(M_{\delta}(\overline{\varphi_\delta^0}) - M_{\delta}(\varphi^0)) \underset{\delta \to 0}{\rightharpoonup} \mathcal{S}^a(0) = 0$. Using a similar argument, we get that $\mathcal{S}^a(M_{\delta}(\varphi^0) - M_{\delta}(\overline{\varphi_{\delta}^0})) \underset{\delta \to 0}{\rightharpoonup} 0$ and the result follows.
\end{proof}

More exactly, we aim to prove the following result.
\begin{ppt}
With the above notation, we have
\[
\iint_{\q} \chi_\omega \left| \mathcal{S}^a(M_\delta(\overline{\varphi_\delta^0})) - \mathcal{S}^a(\overline{\varphi^0}) \right|^2 \ dx \ dt + \varepsilon \| \overline{\varphi_\delta^0} - \overline{\varphi^0} \|^2 \to 0 
\]
when $\delta \to 0$.
\end{ppt}

\begin{proof}
We choose $\psi^0 = \overline{\varphi_\delta^0} - \overline{\varphi^0}$ in~\eqref{eq:optim-delta}-\eqref{eq:optim} and we subtract these relations: 
\begin{align*}
\iint_{\q} \chi_\omega \mathcal{S}^a(M_\delta(\overline{\varphi_\delta^0})) \mathcal{S}^a(M_\delta(\overline{\varphi_\delta^0} - \overline{\varphi^0})) \ dx \ dt
- \iint_{\q} \chi_\omega \mathcal{S}^a(\overline{\varphi^0}) \mathcal{S}^a(\overline{\varphi_\delta^0} - \overline{\varphi^0})  \ dx \ dt
\\
+ \varepsilon \| \overline{\varphi_\delta^0} - \overline{\varphi^0} \|^2
+ (\mathcal{S}^a(M_\delta(\overline{\varphi_\delta^0} - \overline{\varphi^0}))(0) 
- \mathcal{S}^a(\overline{\varphi_\delta^0} - \overline{\varphi^0})(0), y^0) = 0.
\end{align*}
Using the linearity of the equation~\eqref{eq:dual} (hence of $\mathcal{S}^a$) and of  $M_\delta$, the above equality writes as follows:
\begin{align*}
\iint_{\q} \chi_\omega |\mathcal{S}^a(M_\delta(\overline{\varphi_\delta^0}))|^2 \ dx \ dt
+ \iint_{\q} \chi_\omega |\mathcal{S}^a(\overline{\varphi^0})|^2 \ dx \ dt
- \iint_{\q} \chi_\omega \mathcal{S}^a(M_\delta(\overline{\varphi_\delta^0})) \mathcal{S}^a(M_\delta(\overline{\varphi^0}))\ dx \ dt 
\\
- \iint_{\q} \chi_\omega \mathcal{S}^a(\overline{\varphi_\delta^0}) \mathcal{S}^a(\overline{\varphi^0})\ dx \ dt
+ \varepsilon \| \overline{\varphi_\delta^0} - \overline{\varphi^0} \|^2
+ (\mathcal{S}^a(M_\delta(\overline{\varphi_\delta^0} - \overline{\varphi^0}) - (\overline{\varphi_\delta^0} - \overline{\varphi^0}))(0), y^0) = 0.
\end{align*}
Finally, we get
\begin{align*}
\iint_{\q} \chi_\omega |\mathcal{S}^a(M_\delta(\overline{\varphi_\delta^0})) - \mathcal{S}^a(\overline{\varphi^0})|^2 \ dx \ dt
- \iint_{\q} \chi_\omega \mathcal{S}^a(M_\delta(\overline{\varphi_\delta^0})) \mathcal{S}^a(M_\delta(\overline{\varphi^0}) - \overline{\varphi^0}) \ dx \ dt
\\
+ \iint_{\q} \chi_\omega \mathcal{S}^a(\overline{\varphi^0}) \mathcal{S}^a(M_\delta(\overline{\varphi_\delta^0}) - \overline{\varphi_\delta^0}) \ dx \ dt
+ \varepsilon \| \overline{\varphi_\delta^0} - \overline{\varphi^0} \|^2
+ (\mathcal{S}^a(M_\delta(\overline{\varphi_\delta^0} - \overline{\varphi^0}) - (\overline{\varphi_\delta^0} - \overline{\varphi^0}))(0), y^0) = 0.
\end{align*}
The result follows applying Lemma~\ref{lem:weak-limit-delta}.
\end{proof}

We now give the proof of Theorem~\ref{thm:smooth-controls-linear-heat}.

\begin{proof}[Proof of Theorem~\ref{thm:smooth-controls-linear-heat}]
The existence of approximate controls $\varphi \in C^\infty(\Q)$ is obtained by minimizing the functional $J^a_\delta$ applying the standard HUM method, the regularity being derived from the usual regularity in the linear parabolic case (see, for example, \cite[Theorem 10.1]{brezis}). We therefore focus on proving the Lipschitz $L^2$ continuity with respect to the diffusion coefficient $a$.

Let us consider two diffusion coefficients $a$ and $b$ in $C^\infty(\overline{Q_T})$ verifying~\eqref{eq:condition_diffusion-coeff}. Then, we denote $\varphi_a = \mathcal{S}^a(M_\delta(\overline{\varphi_{a, \delta}^0}))$ and $\varphi_b = \mathcal{S}^b(M_\delta(\overline{\varphi^0_{b, \delta}}))$ with $\overline{\varphi^0_{a, \delta}}$ being the minimum of $J^a_\delta$ and $\overline{\varphi^0_{b, \delta}}$ being the minimum of $J^b_\delta$. Writing $w := \varphi_a - \varphi_b$, we get that $w$ satisfies the following equation:

\begin{equation}\label{eq:dual-regul-datum-w}
\left\{\begin{array}{ll}
\partial_tw + \mathrm{div}\left(a(t,x)\nabla w\right) = -\mathrm{div}\left((a-b)\nabla\varphi_b\right)& \text{ in }\; \Q\\
w = 0& \text{ on }\; \Sig\\
w(T) = M_\delta(\overline{\varphi_{a, \delta}^0} - \overline{\varphi_{b, \delta}^0})& \text{ in }\; \Omega.
\end{array}\right.
\end{equation}

An energy estimate over \eqref{eq:dual-regul-datum-w} leads:

\begin{equation}\label{eq:first-estimate-lipschitz}
\frac{1}{2}\lVert w(T) \rVert^2 + (\rho_{\star} - s)\lVert w \rVert_{L^2((0,T),H_0^1(\Omega))}^2 \leq \frac{1}{4s}\lVert \varphi_b \rVert_{W^{1,\infty}(\Omega)}^2\lVert a - b \rVert_{L^2(\Q)}^2 + \frac{1}{2}\lVert M_\delta(\overline{\varphi_{a, \delta}^0} - \overline{\varphi_{b, \delta}^0}) \rVert^2.
\end{equation}

According to the Young's inequality for convolution (see \cite[Theorem 4.15.]{brezis}):

\begin{equation}\label{eq:young-regul-imply-non-regul}
\lVert M_\delta(\overline{\varphi_{a, \delta}^0} - \overline{\varphi_{b, \delta}^0}) \rVert \leq \lVert \rho_{\delta} \rVert_{L^1(\Omega)} \lVert \overline{\varphi_{a,\delta}^0} - \overline{\varphi_{b,\delta}^0} \rVert \leq \lVert \overline{\varphi_{a,\delta}^0} - \overline{\varphi_{b,\delta}^0} \rVert
\end{equation}

where $(\rho_{\delta})_{\delta > 0}$ is the mollifier used to define $M_\delta$. Then, we get from Euler-Lagrange formula:

\begin{align}
\label{eq:euler-lagrange-control-hum-regul-a}
\iint_{\q}\chi_{\omega}\varphi_a \mathcal{S}^a(M_\delta(\psi^0)) \;dx\,dt + \varepsilon(\overline{\varphi_{a,\delta}^0},\psi^0) +(\mathcal{S}^a(M_\delta(\psi^0))(0),y^0) = 0\\
\label{eq:euler-lagrange-control-hum-regul-b}
\iint_{\q}\chi_{\omega}\varphi_b \mathcal{S}^b(M_\delta(\psi^0)) \;dx\,dt + \varepsilon(\overline{\varphi_{b,\delta}^0},\psi^0) +(\mathcal{S}^b(M_\delta(\psi^0))(0),y^0) = 0.
\end{align}

Subtracting the relations \eqref{eq:euler-lagrange-control-hum-regul-a}-\eqref{eq:euler-lagrange-control-hum-regul-b}, we obtain:

\begin{align}
\iint_{\q}\chi_{\omega}(\varphi_a\mathcal{S}^a(M_\delta(\psi^0)) - \varphi_b\mathcal{S}^b(M_\delta(\psi^0)))\;dx\,dt + \varepsilon(\overline{\varphi_{a,\delta}^0} - \overline{\varphi_{b,\delta}^0},\psi^0) \nonumber \\
+(\mathcal{S}^a(M_\delta(\psi^0))(0) - \mathcal{S}^b(M_\delta(\psi^0))(0),y^0) = 0. \label{eq:on-the-way-to-lipschitz}
\end{align}

Now, one can write:
\begin{align}
    \iint_{\q}\chi_{\omega}(\varphi_a\mathcal{S}^a(M_\delta(\psi^0)) - \varphi_b\mathcal{S}^b(M_\delta(\psi^0)))\;dx\,dt =
    \iint_{\q}\chi_{\omega}\lvert \varphi_a - \varphi_b \rvert^2\;dx\,dt \nonumber \\
    + \iint_{\q}\chi_{\omega}(\varphi_a - \varphi_b)(\mathcal{S}^a(M_\delta(\psi^0)) - \varphi_a + \varphi_b)\; dx\,dt
    + \iint_{\q}\chi_{\omega}\varphi_b(\mathcal{S}^a(M_\delta(\psi^0)) - \mathcal{S}^b(M_\delta(\psi^0)))\; dx\,dt. \label{eq:decomposition-adding-substracting}
\end{align}

Setting $\psi^0 := \overline{\varphi_{a,\delta}^0} - \overline{\varphi_{b,\delta}^0}$ into \eqref{eq:on-the-way-to-lipschitz} and using \eqref{eq:decomposition-adding-substracting} leads to:

\begin{align}
\iint_{\q}\chi_{\omega}\lvert \varphi_a - \varphi_b \rvert^2\;dx\,dt + \varepsilon\lVert \overline{\varphi_{a,\delta}^0} - \overline{\varphi_{b,\delta}^0} \rVert^2 + (\mathcal{S}^a(M_\delta(\overline{\varphi_{a,\delta}^0} - \overline{\varphi_{b,\delta}^0}))(0) - \mathcal{S}^b(M_\delta(\overline{\varphi_{a,\delta}^0} - \overline{\varphi_{b,\delta}^0}))(0),y^0) = \nonumber \\
-\iint_{\q}\chi_{\omega}(\varphi_a - \varphi_b)(\mathcal{S}^a(M_\delta(\overline{\varphi_{a,\delta}^0} - \overline{\varphi_{b,\delta}^0})) - \varphi_a + \varphi_b)\; dx\,dt \nonumber \\
    -\iint_{\q}\chi_{\omega}\varphi_b(\mathcal{S}^a(M_\delta(\overline{\varphi_{a,\delta}^0} - \overline{\varphi_{b,\delta}^0})) - \mathcal{S}^b(M_\delta(\overline{\varphi_{a,\delta}^0} - \overline{\varphi_{b,\delta}^0})))\; dx\,dt. \label{eq:decomposition-adding-substracting-2}
\end{align}

Here, writing $\beta := \mathcal{S}^a(M_\delta(\overline{\varphi_{a,\delta}^0} - \overline{\varphi_{b,\delta}^0})) - \mathcal{S}^b(M_\delta(\overline{\varphi_{a,\delta}^0} - \overline{\varphi_{b,\delta}^0}))$, we have that it solves:

\begin{equation}\label{eq:difference-beta}
\left\{\begin{array}{ll}
\partial_t\beta+ \mathrm{div}\left(a\nabla\beta\right) = \mathrm{div}\left((a-b)\nabla \mathcal{S}^b(M_\delta(\overline{\varphi_{a,\delta}^0} - \overline{\varphi_{b,\delta}^0})\right)& \text{ in }\; \Q\\
\beta = 0& \text{ on }\; \Sig\\
\beta(T) = 0& \text{ in } \Omega.
\end{array}\right.
\end{equation}

Testing against $\beta$ into the weak formulations of \eqref{eq:difference-beta} leads to, after applying Hölder's and Young's inequality for $0 < s < \rho_\star$, we get:

\begin{equation}\label{eq:difference-beta-2}
    \frac{1}{2}\lVert \beta(0) \rVert^2 + (\rho_\star - s)\iint_{\Q}\lvert \nabla\beta \vert^2\;dx\,dt \leq \frac{1}{4s}\lVert \mathcal{S}^b(M_\delta(\overline{\varphi_{a,\delta}^0} - \overline{\varphi_{b,\delta}^0})\rVert_{W^{1,\infty}(\Q)}^2\lVert a-b \rVert_{L^2(\Q)}^2.
\end{equation}

From Poincaré's inequality \eqref{eq:difference-beta-2} leads to:

\begin{equation}\label{eq:difference-beta-3}
    \lVert \beta \rVert_{L^2(\Q)} \leq \left(\frac{\lambda_1(\Omega)^{-1}}{4s(\rho_\star - s)}\right)^{\frac{1}{2}}\lVert \mathcal{S}^b(M_\delta(\overline{\varphi_{a,\delta}^0} - \overline{\varphi_{b,\delta}^0})\rVert_{W^{1,\infty}(\Q)}\lVert a-b \rVert_{L^2(\Q)},
\end{equation}

where $\lambda_1(\Omega)^{-1}$ is the sharp Poincaré constant, and $\rho_{\star}$ comes from \eqref{eq:condition_diffusion-coeff}. Using \eqref{eq:difference-beta-3} into \eqref{eq:decomposition-adding-substracting-2}, then Young's inequality and an energy estimate, we then get for $0< s < 1$:

\begin{align}\label{eq:on-the-way-to-lipschitz-3}
(1-s)\iint_{\q}\chi_{\omega}\lvert \varphi_a - \varphi_b \rvert^2\;dx\,dt + \varepsilon\lVert \overline{\varphi_{a, \delta}^0} - \overline{\varphi_{b,\delta}^0} \rVert^2 \leq \frac{1}{4s}\lVert \mathcal{S}^a(M_\delta(\overline{\varphi_{a,\delta}^0} - \overline{\varphi_{b,\delta}^0})) - \varphi_a + \varphi_b \rVert^2_{L^2(\q)} \nonumber\\
 + \lVert \varphi_b \rVert_{L^2(\q)}\lVert \mathcal{S}^a(M_\delta(\overline{\varphi_{a,\delta}^0} - \overline{\varphi_{b,\delta}^0})) - \mathcal{S}^b(M_\delta(\overline{\varphi_{a,\delta}^0} - \overline{\varphi_{b,\delta}^0}))\rVert_{L^2(\q)} \nonumber \\
 + \lVert y^0 \rVert\left(\frac{\lambda_1(\Omega)^{-1}}{4s(\rho_\star - s)}\right)^{\frac{1}{2}}\lVert \mathcal{S}^b(M_\delta(\overline{\varphi_{a,\delta}^0} - \overline{\varphi_{b,\delta}^0})\rVert_{W^{1,\infty}(\Q)}\lVert a-b \rVert_{L^2(\Q)}.
\end{align}

Now, setting $\mathsf{W} = \mathcal{S}^a(M_\delta(\psi^0)) -\varphi_a + \varphi_b$ and $\mathsf{w}= \mathcal{S}^a(M_\delta(\psi^0)) - \mathcal{S}^b(M_\delta(\psi^0))$, we get that these respectively solve:

\begin{equation}\label{eq:difference-W}
\left\{\begin{array}{ll}
\partial_t\mathsf{W} + \mathrm{div}\left(a\nabla\mathsf{W}\right) = \mathrm{div}((a-b)\nabla \varphi_b)& \text{ in }\; \Q\\
\mathsf{W} = 0& \text{ on }\; \Sig\\
\mathsf{W}(T) = 0& \text{ in } \Omega
\end{array}\right.
\end{equation}

\begin{equation}\label{eq:difference-w}
\left\{\begin{array}{ll}
\partial_t\mathsf{w} + \mathrm{div}\left(a\nabla \mathsf{w}\right) = - \mathrm{div}((a-b)\nabla\mathcal{S}^b(\psi^0))&\text{ in }\; \Q\\
\mathsf{w} = 0& \text{ on }\; \Sig\\
\mathsf{w}(T) = 0&\text{ in }\; \Omega
\end{array}\right.
\end{equation}

and following exactly the same argument as for \eqref{eq:difference-beta-3} leads to, for some $0 < s < \rho_\star$:

\begin{equation}\label{eq:energy-estimates-lipschitz}
\left\{\begin{array}{l}
\lVert \mathsf{W} \rVert_{L^2(\q)} \leq \left(\frac{\lambda_1(\Omega)^{-1}}{4\rho_{\star}(1-s)s}\right)^{\frac{1}{2}} \lVert \varphi_b \rVert_{W^{1,\infty}(\Omega)}^2\lVert a - b \rVert_{L^2(\Q)}^2\\
\\
\lVert \mathsf{w} \rVert_{L^2(\q)} \leq \left(\frac{\lambda_1(\Omega)^{-1}}{4\rho_{\star}(1-s)s}\right)^{\frac{1}{2}} \lVert \mathcal{S}^b(\psi^0) \rVert_{W^{1,\infty}(\Omega)}^2\lVert a - b \rVert_{L^2(\Q)}^2.
\end{array}\right.
\end{equation}

Combining \eqref{eq:on-the-way-to-lipschitz-3} to \eqref{eq:energy-estimates-lipschitz}, finally leads to the existence of a positive constant $C > 0$ such that:

\begin{align}
\varepsilon\lVert \overline{\varphi_{a,\delta}^0} - \overline{\varphi_{b,\delta}^0}\rVert^2 &\leq (1-s)\iint_{\q}\chi_{\omega}\lvert \varphi_a - \varphi_b \rvert^2\;dx\,dt + \varepsilon\lVert \overline{\varphi_{a,\delta}^0} - \overline{\varphi_{b,\delta}^0} \rVert^2\nonumber \\
 &\leq C(\lVert a - b \rVert_{L^2(\Q)}^2 + \lVert a - b \rVert_{L^2(\Q)}). \label{eq:arrived-at-Lipschitz}
\end{align}

By dividing each member of \eqref{eq:arrived-at-Lipschitz} by $\varepsilon > 0$ and combining this with \eqref{eq:young-regul-imply-non-regul} and \eqref{eq:first-estimate-lipschitz}, we deduce the result.
\end{proof}

At this point, it is worth noting that the Lipschitz constant in $L^2$ of the control is strongly dependent on the parameter $\varepsilon > 0$. Indeed, the Lipschitz constant thus obtained explodes as $\varepsilon$ tends towards $0$. Consequently, there is no reason to conclude that the exact control, of minimal $L^2$ norm, depends continuously in the $L^2$ sense on the diffusion coefficient.

\begin{remark} Following \cite{sgefc}, one can build exact control for a slightly modified functional given by

\begin{equation*}
J_{\varepsilon}(\varphi^0) = \frac{1}{2}\iint_{\q}\chi_{\omega}\lvert \varphi \rvert^2\;dx\, dt + \varepsilon\lVert \varphi^0 \rVert + (\varphi(0),y^0).
\end{equation*}
However, this method fails to build approximate controls which are Lipschitz continuous in the $L^2$ sense. This being due to the lack of monotonicity of the sign function obtained in the associated Euler-Lagrange equality.
\end{remark}

\section{Controllability of the quasilinear equation}\label{sec-control-quasilinear}

In order to extend the controllability properties of the linear equation~\eqref{eq:1} to the quasilinear equation~\eqref{eq:0}, we aim to apply a fixed point theorem. In this purpose, we first consider the strategy proposed in Section~\ref{sec:control-linear}, which allows to obtain regular approximate controls for the equation~\eqref{eq:1} which are Lipschitz continuous with respect to the diffusion coefficient $a$. The existence of  approximate regular controls for the linear equation~\eqref{eq:1} makes possible to define an application associating to the diffusion coefficient $a$ the quantity $F_\delta(|\nabla u_{a, \delta}|)$ from a bounded closed convex set with values in itself, $u_{a, \delta}$ being the controlled solution of the regularized version of~\eqref{eq:1} and $F_\delta$ being a regularization of the function $F$ still verifying hypotheses \ref{A1}--\ref{A3}.
The objective is then to show the continuity of such applications on some weakly sequentially compact sets, in order to apply a suitable fixed point theorem. Remark that, since we aim to prove only a null approximate controllability result, it is not necessary to take the limit with respect to the regularization parameter $\delta$. Here and henceforth we denote by ``$*$'' the convolution product.

We define a regularisation process $R_{\delta} : L^1(\Q) \rightarrow C^{\infty}(\overline{\Q})$. More exactly, for every $g \in L^1(\Q)$ we define $R_{\delta}(g)$ by
\begin{equation}\label{eq:defi-regul-process-R}
    R_{\delta}(g) = \nu_{\delta}* (\chi_{\delta}g) + \delta,
\end{equation}
where $(\nu_{\delta})_{\delta}$ is a mollifier, $ \chi_{\delta} :\mathbb{R}^{N+1} \rightarrow \mathbb{R}$ is a smooth cutoff function with $\textrm{supp}(\chi_\delta)=\Q$. We can see that $R_{\delta}(g) \underset{\delta \rightarrow 0}{\longrightarrow} g$ in $L^1(\Q)$ for every $g \in L^1(\Q)$. 

From now on, we should denote $L^2_+(\Q)$ the subspace of non negative functions in $L^2(\Q)$. Then, we consider the following bounded convex closed set
\begin{equation}
    \label{eq:Keps}
    K_\delta = \left\{
    f \in L^2_+(Q_T) \text{ satisfying } \|f\|^2_{L^2(Q_T)} \le 
    2(C_L + 1)\left(\frac{1}{\delta} + \frac{C_0}{2\delta^2}\right)\lVert y^0 \rVert^2 + \delta\lvert \Omega \rvert T
    \right\},
\end{equation}
where $C_L$ is the Lipschitz constant of $F$ from assumption \ref{A1} and $C_0$ is the observability constant given in~\eqref{eq:observability}.
Moreover, let us observe that $R_{\delta}$ is continuous over $L^2_+(\Q)$. Taking $a \in K_{\delta}$ and  $h \in L^2(\Q)$ such that $a+h \in K_{\delta}$, we can write:
\begin{equation}\label{eq:est-convol-ah}
    R_{\delta}(a+h) - R_{\delta}(a) = \nu_{\delta}* (\chi_{\delta}h)
\end{equation}
and the continuity of $R_\delta$ follows from Hölder's inequality.

Let us now define the function
$G_\delta : K_{\delta} \to L^2(\Q)$  by
\begin{equation}\label{eq:Geps}
G_{\delta}(a) = F(\lvert \nabla v_{a,\delta} \rvert),
\end{equation}

where $v_{a,\delta}$ is the weak solution to:

\begin{equation}\label{eq:1-regul}
\left\{
\begin{array}{ll}
\partial_t v_{a,\delta} - {\rm div}\left(R_{\delta}(a) (t,x)\nabla v_{a,\delta}\right) = \chi_{\omega}\varphi_a& \text{ in } \Q\\
v_{a,\delta} = 0& \text{ on } \Sig \\
v_{a,\delta}(0) = M_\delta(y^0)& \text{ in } \Omega,\\
\end{array}
\right.
\end{equation}
with $\varphi_a \in C^{\infty}(\overline{\Q})$ being the approximate control provided by Theorem~\ref{thm:smooth-controls-linear-heat} applied to \eqref{eq:1-regul} which is nothing else than an alternative version of \eqref{eq:1} with a regularised operator $L_{\delta} : u \mapsto -\mathrm{div}\left(R_{\delta}(a)\nabla u\right)$.

Let us recall the following fixed point theorem (see \cite{latrach-taoudi-zeghal-06} for a proof). 

\begin{thm}[{\cite[Theorem 2.1]{latrach-taoudi-zeghal-06}}]\label{thm:fixed-point-wsc} Let $K$ be a closed convex subset of a Banach space $Y$. Let us consider $G : K \rightarrow K$ such that for all sequence $(a_n)_n \subset K$ which converges weakly toward $a$, then $(G(a_n))_n$ admits a subsequence which converges strongly toward $G(a)$. Hence, if $G$ is continuous and $G(K)$ is weakly compact, $G$ admits a fixed point.
\end{thm}

Before proving the Theorem~\ref{thm:global-approximate-control-quasi}, let us prove the following lemma.

\begin{lem}\label{lem:continuity-F} For every $\delta > 0$ the application $G_{\delta} : K_{\delta} \to L^2(\Q)$ defined by~\eqref{eq:Geps} is continuous and verifies $G(K_\delta) \subset K_\delta$.
\end{lem}

\begin{proof} 
First, let us show the fact that $G_{\delta}(K_\delta) \subset K_\delta$. Since $K_{\delta}$ contains $0_{L^2(\Q)}$, from~\ref{A1} one can write:

\begin{align}\label{eq:bound-v-delta-1}
    \lVert G_{\delta}(a) - G_{\delta}(0) \rVert_{L^2(\Q)}^2 &= \lVert F(\lvert \nabla v_{a,\delta} \rvert) - F(\lvert \nabla v_{0,\delta} \rvert) \rVert_{L^2(\Q)}^2 \nonumber\\
    &\leq C_L \lVert v_{a,\delta} - v_{0,\delta} \rVert_{L^2((0,T),H_0^1(\Omega))}^2.
\end{align}

Now, we point out that for every $a \in K_{\delta}$, $\lVert v_{a,\delta} \rVert_{L^2((0,T),H_0^1(\Omega))}$ is bounded by $\lVert M_\delta(y^0) \rVert\sqrt{\frac{1}{\delta} + \frac{C_0}{2\delta^2}}$, this bound following from the energy estimate:

\begin{equation}\label{eq:step}
    \frac{1}{2}\lVert v_{a,\delta} \rVert_{L^2(\Q)}^2 + \frac{\delta}{2}\lVert v_{a,\delta} \rVert_{L^2((0,T),H_0^1(\Omega))}^2 \leq \frac{1}{2}\lVert M_\delta(y^0) \rVert^2 + \frac{1}{4\delta}\lVert \chi_{\omega}\varphi_a\rVert_{L^2(\q)}^2\, ,
\end{equation}

since testing against $\varphi^0$ into the Euler-Lagrange equation associated to \eqref{eq:reg-HUM-functional}, then applying the observability inequality as given in Proposition~\ref{ppt:observability} and Young's inequality leads to the inequality:

\begin{equation}\label{eq:observability-consequence}
\lVert \chi_{\omega}\varphi_a \rVert_{L^2(\q)}^2 \leq C_0\lVert M_\delta(y^0) \rVert^2\, ,
\end{equation}

and combining \eqref{eq:step} together with \eqref{eq:observability-consequence} implies

\begin{equation}\label{eq:bound-v-eps}
    \lVert v_{a,\delta} \rVert_{L^2((0,T),H_0^1(\Omega))}^2 \leq \frac{1}{\delta}\lVert M_\delta(y^0) \rVert^2 + \frac{1}{2\delta^2}\lVert \chi_{\omega}\varphi_a\rVert_{L^2(\q)}^2 \leq \left(\frac{1}{\delta} + \frac{C_0}{2\delta^2}\right)\lVert M_\delta(y^0) \rVert^2.
\end{equation}

Now, from \eqref{eq:bound-v-delta-1}, we can write:

\begin{equation}\label{eq:G-eps-wd-1}
    \lVert G_{\delta}(a) \rVert_{L^2(\Q)}^2 \leq C_L\left(\lVert v_{a,\delta} \rVert_{L^2((0,T),H_0^1(\Omega))}^2 + \lVert v_{0,\delta} \rVert_{L^2((0,T),H_0^1(\Omega))}^2\right) + \lVert G_{\delta}(0) \rVert_{L^2(\Q)}^2.
\end{equation}

On the other hand, we have:

\begin{equation}\label{eq:est-G-eps-0}
    \lVert G_{\delta}(0) \rVert_{L^2(\Q)}^2 \leq \lVert v_{0,\delta} \rVert_{L^2((0,T),H_0^1(\Omega)}^2 + \delta \lvert \Omega \rvert T.
\end{equation}

Hence, combining \eqref{eq:est-G-eps-0}, \eqref{eq:G-eps-wd-1} and \eqref{eq:bound-v-eps}, we get that $G_\delta(K_\delta) \subset K_\delta$.

Let us now show that $G_{\delta}$ is continuous. Since $F$ is globally Lipschitz from assumption \ref{A1}, one have: 
\begin{equation}\label{eq:continuity-F-1}
\lVert G_{\delta}(a+h) - G_{\delta}(a) \rVert_{L^2(\Q)}^2 = \|F(\lvert \nabla v_{a+h,\delta} \rvert) - F(\lvert \nabla v_{a,\delta} \rvert)\|_{L^2(\Q)}^2\leq C_L\lVert \nabla v_{a+h,\delta} - \nabla v_{a,\delta} \rVert_{L^2(\Q)}^2.
\end{equation}

Then, we have that $w_{a,h} := v_{a+h,\delta} - v_{a,\delta}$ solves:

\begin{equation}\label{eq:diff-continuity-F}
\left\{
\begin{array}{ll}
\partial_tw_{a,h} - {\rm div}\left(R_{\delta}(a+h)\nabla w_{a,h}\right) = \chi_{\omega}\left(\varphi_{a+h} - \varphi_a\right) -{\rm div}\left(\left(R_{\delta}(a + h) - R_{\delta}(a)\right) \nabla v_{a+h}\right) & \text{ in } \Q\\
w_{a,h} = 0& \text{ on } \Sig\\
w_{a,h}(0) = 0& \text{ in } \Omega.
\end{array}
\right.
\end{equation}

Hence, an energy estimate over \eqref{eq:diff-continuity-F} leads to, using parametrized Young's inequality and Poincaré's inequality:
\begin{align*}
\frac{1}{2}\lVert w_{a,h} \rVert_{L^{\infty}((0,T),L^2(\Omega))}^2 + \frac{\delta}{2}\lVert w_{a,h} \rVert_{L^2((0,T),H_0^1(\Omega))}^2 \leq & \frac{1}{\delta}\lVert \varphi_{a+h} - \varphi_a \rVert_{L^2(\Q)}^2 \\
& + \frac{1}{\delta}\lVert \nabla v_{a+h} \rVert_{L^{\infty}(\Q)}^2\lVert R_{\delta}(a+h) - R_{\delta}(a) \rVert_{L^2(\Q)}^2.
\end{align*}

Using Theorem~\ref{thm:smooth-controls-linear-heat} combined to~\eqref{eq:continuity-F-1}, and since, for a $k$ large enough with respect to $N$, $H_0^k(\Omega)$ is continuous embedded in  $W_0^{1,\infty}(\Omega)$, we get that:
\begin{equation}
    \lVert G_{\delta}(a+h) - G_{\delta}(a) \rVert_{L^2(\Q)}^2 \leq \frac{C(\delta)}{\delta^2}\left(\lVert h \rVert_{L^2(\Q)}^2 + \lVert M_\delta(y^0) \rVert_k^2\lVert R_{\delta}(a + h) - R_{\delta}(a) \rVert_{L^2(\Q)}^2 \right),
\end{equation}
which proves the lemma. 
\end{proof}

\begin{remark}\label{rem:double-regularization}
  We can avoid the global regularity assumptions over $F$ as in \ref{A1} by considering that
  \[
    F \in W^{1,\infty}_\text{loc}(\mathbb{R}_+) \cap L^2_\text{loc}(\mathbb{R}_+) \cap
    L^{\frac{p}{p-1}}_\text{loc}(\mathbb{R}_+).
  \]
  It is then necessary to introduce an additional regularization process. First, for every $\delta > 0$ we define a regularisation process  $r_{\delta} : L^1_\text{loc}(\mathbb{R}_+) \cap W^{1,\infty}_\text{loc}(\mathbb{R}_+) \to C^{\infty}(\mathbb{R}_+) \cap W^{1,\infty}(\mathbb{R}_+)$ by:
\begin{equation}\label{def:real-line-regul-process}
   r_{\delta}(F) := \zeta_{\delta}*(\sigma_{\delta}F) + \delta\, ,
\end{equation}
for every $F \in L^1(\mathbb{R}_+) \cap W^{1,\infty}(\mathbb{R}_+)$,
where $(\zeta_{\delta})_{\delta}$ is a mollifier and $ \sigma_{\delta} : \mathbb{R} \rightarrow \mathbb{R}$ is the smooth cutoff function satisfying:
\begin{equation}\label{eq:cutoff-smooth-delta}
    \sigma_{\delta} = \left\{
    \begin{array}{ll}
       1  & \text{ in } \left[\delta, \frac{1}{\delta}\right] \\ 
        0 &  \text{ in } (-\infty,0] \cup \left[\frac{1}{\delta} + \delta, + \infty\right).
    \end{array}\right.
\end{equation}

Then, we see that such a regularization process also holds over $L^1_{\mathrm{loc}}(\mathbb{R}_+) \cap W^{1,\infty}(\mathbb{R}_+\backslash\{0\})$, in the sense that for every $F \in L^1_{\mathrm{loc}}(\mathbb{R}_+) \cap W^{1,\infty}(\mathbb{R}_+\backslash\{0\})$, such an $r_\delta(F)$ leads to a globally Lipschitz function, \emph{i.e.}, there exists $C_{L, \delta} > 0$ such that
\begin{equation}
    \label{eq:lip-reps}
    \lvert r_\delta(F)(t) - r_\delta(F)(s) \rvert \le C_{L,\delta} \lvert t - s\rvert,
\end{equation}
for every $t,\ s > 0$.
\end{remark}

\begin{remark}
 We point out that is essential here to consider solutions to the regularized equation~\eqref{eq:1-regul}, since the space of functions which are essentially positively lower and upper bounded do not give rise to regular enough solutions of \eqref{eq:1} (namely, at least Hölder continuous) since we can construct discontinuous solutions with respect to the space variable of \eqref{eq:1} for some diffusion coefficient in this space, given by Serrin's example (see \cite{serrin-63}).
\end{remark}

We are now able to prove Theorem~\ref{thm:global-approximate-control-quasi}.

\begin{proof}[Proof of Theorem~\ref{thm:global-approximate-control-quasi}] Let $\delta > 0$ to choose later.
In order to apply Theorem~\ref{thm:fixed-point-wsc}, we first show that if $(a_n)_n$ is a sequence which converges weakly toward $a$, then $G_{\delta}(a_n)$ converges weakly, up to a subsequence, toward $G_{\delta}(a)$. First, let us observe that the weak convergence of $(a_n)_n$ implies that $R_{\delta}(a_n) \underset{n \rightarrow +\infty}{\longrightarrow} R_{\delta}(a)$ strongly in $K_{\delta}$, by definition of $R_{\delta}$. Also, arguing by continuity, one can see that the associated controls provided by Theorem~\ref{thm:smooth-controls-linear-heat} in \eqref{eq:1-regul} leads to 
\[
\lVert \chi_{\omega}\varphi_{R_{\delta}(a_n)} - \chi_{\omega}\varphi_{R_{\delta}(a)}\rVert_{L^2(\Q)} \underset{n \rightarrow +\infty}{\longrightarrow} 0.\]
From this, an energy estimate 
leads to, setting $w_n := v_{a,\delta} - v_{a_n,\delta}$, where $v_{a,\delta}$ and $v_{a_n,\delta}$ are respectively solutions to \eqref{eq:1-regul} associated to the diffusion coefficient $R_{\delta}(a)$ and $R_{\delta}(a_n)$ and to the controls $\varphi_{R_\delta(a)}$ and $\varphi_{R_\delta(a_n)}$, respectively:

\begin{align}\label{eq:energy-est-continuity-fixed-point-2}
\frac{1}{2}\lVert w_n \rVert_{L^{\infty}((0,T),L^2(\Omega))}^2 + \frac{\delta}{2}\lVert w_n \rVert_{L^2((0,T),H_0^1(\Omega))}^2 \leq& \frac{1}{\delta}\lVert \chi_\omega( \varphi_{R_{\delta}(a)} - \varphi_{R_{\delta}(a_n)}) \rVert_{L^2(\q)}^2 \nonumber \\
&+ \frac{C\lVert y^0 \rVert_k}{\delta}\lVert R_{\delta}(a_n) - R_{\delta}(a) \rVert_{L^2(\Q)}^2.
\end{align}

Then, we can write, from \eqref{eq:continuity-F-1}:

\begin{equation}\label{eq:continuity-F-2}
\lVert G_{\delta}(a_n) - G_{\delta}(a) \rVert_{L^2(\Q)}^2 = \|F(\lvert \nabla v_{a_n,\delta} \rvert) - F(\lvert \nabla v_{a,\delta} \rvert)\|_{L^2(\Q)}^2\leq C_L\lVert w_n \rVert_{L^2(0, T, H_0^1(\Omega))}^2.
\end{equation}

Then, \eqref{eq:energy-est-continuity-fixed-point-2} combined with \eqref{eq:continuity-F-2} leads to the fact that $\lVert G_{\delta}(a_n) - G_{\delta}(a) \rVert_{L^2(\Q)} \underset{n \rightarrow + \infty}{\longrightarrow} 0$. Applying now Lemma~\ref{lem:continuity-F}, we get from Theorem~\ref{thm:fixed-point-wsc} that $G_{\delta}$ admits a unique fixed point in $K_{\delta}$. Namely, we get that in the solution to the equation

\begin{equation}\label{eq:quasi-reg}
\left\{
\begin{array}{ll}
\partial_t\mathsf{v}_{\delta} - {\rm div}\left(R_{\delta}\left(F(\lvert \nabla \mathsf{v}_{\delta} \rvert)\right)\nabla \mathsf{v}_{\delta}\right) = \chi_{\omega}\varphi& \text{ in } \Q\\
\mathsf{v}_{\delta} = 0& \text{ on } \Sig \\
\mathsf{v}_{\delta}(0) = M_\delta(y^0)& \text{ in } \Omega.\\
\end{array}
\right.
\end{equation}
 $\varphi$ can be chosen as an approximate control of \eqref{eq:quasi-reg}, from Theorem~\ref{thm:smooth-controls-linear-heat}. For the sake of simplicity, we denote $R_{\delta}(F(\lvert \cdot \rvert))$ as $F_{\delta}(\lvert \cdot\rvert)$. Next, we denote $w := y-\mathsf{v}_{\delta}$ with $y$ the solution to~\eqref{eq:0}. Writing:

\begin{equation}\label{eq:equality-monotonicity-approximation}
F(\lvert \nabla y \rvert)\nabla y - F_{\delta}(\lvert \nabla \mathsf{v}_{\delta} \rvert) \nabla \mathsf{v}_{\delta} = F(\lvert \nabla y \rvert)\nabla y - F(\lvert \nabla \mathsf{v}_{\delta} \rvert)\nabla \mathsf{v}_{\delta} + (F(\lvert \nabla \mathsf{v}_{\delta} \rvert) - F_{\delta}(\lvert \nabla \mathsf{v}_{\delta} \rvert))\nabla \mathsf{v}_{\delta},
\end{equation}

an energy estimate leads to:

\begin{align}
    &\frac{1}{2}\lVert w(T) \rVert^2 + \iint_{\Q}\left(F(\lvert \nabla y \rvert)\nabla y - F(\lvert \nabla \mathsf{v}_{\delta} \rvert)\nabla \mathsf{v}_{\delta}\right)\cdot(\nabla y - \nabla \mathsf{v}_{\delta}) \;dx dt \nonumber\\
    &+ \iint_{\Q}(F(\lvert \nabla \mathsf{v}_{\delta} \rvert)\nabla\mathsf{v}_{\delta} - F_{\delta}(\lvert \nabla \mathsf{v}_{\delta} \rvert)\nabla\mathsf{v}_{\delta})\cdot(\nabla y - \nabla \mathsf{v}_{\delta}) \; dx \,dt = \frac{1}{2}\lVert M_\delta(y^0) - y^0 \rVert^2.\label{eq:est-control-quasi-1}
\end{align}

Then, \eqref{eq:est-control-quasi-1} leads, using the monotonicity of the operator (assumption (A3) see \cite[Chapitre 2, section 1.3.]{lions-quelques-resolutions} and \cite[section 25.3]{zeidler-nfa2b-90}) to:

\begin{align}
  \frac{1}{2}\lVert w(T) \rVert^2 &\leq \frac{1}{2}\lVert w(T) \rVert^2 + \iint_{\Q}\left(F(\lvert \nabla y \rvert)\nabla y - F(\lvert \nabla \mathsf{v}_{\delta} \rvert)\nabla \mathsf{v}_{\delta}\right)\cdot(\nabla y - \nabla \mathsf{v}_{\delta}) \;dx dt \nonumber\\
  & \leq \left| \iint_{\Q}(F(\lvert \nabla \mathsf{v}_{\delta} \rvert)\nabla\mathsf{v}_{\delta} - F_{\delta}(\lvert \nabla \mathsf{v}_{\delta} \rvert)\nabla\mathsf{v}_{\delta})\cdot(\nabla y - \nabla \mathsf{v}_{\delta}) \; dx \,dt\right| + \frac{1}{2}\lVert M_\delta(y^0) - y^0 \rVert^2\nonumber\\
  &=\left|\iint_{\Q}(\nabla\Phi(\lvert \nabla\mathsf{v}_{\delta} \rvert) - \nabla\Phi_{\delta}(\lvert\nabla\mathsf{v}_{\delta}\rvert))\cdot \nabla(y - \mathsf{v}_{\delta})\;dx\,dt\right|  + \frac{1}{2}\lVert M_\delta(y^0) - y^0 \rVert^2.\label{eq:controllability-quasilinear}
\end{align}

It remains to prove that the first term in the right hand side goes to zero as $\delta$ does. First, let us remark that we can write:

\begin{equation}\label{eq:estimate-convolution-epsilon}
\begin{array}{rl}
	\nabla\Phi_{\delta}(\lvert x \rvert) &:= F_{\delta}(\lvert x \rvert)x\\
    &=R_{\delta}(F)(\lvert x \rvert)x\\
    &=\left(\nu_{\delta}*(\chi_{\delta}F)(\lvert x\rvert) + \delta\right)x\\
    &=\nu_{\delta}*(\chi_{\delta}F)(\lvert x\rvert)x + \delta x\\
    &= \nu_{\delta}*(\chi_{\delta}F(\lvert x\rvert)x) + \delta x\\
    &= \nabla(\Phi)^{\delta}(\lvert x \rvert) + \delta x.\\
    \end{array}
  \end{equation}

Here, we denoted by $(\Phi)^{\delta}$ a regularization (by mollification) of $\Phi$. Combining~\eqref{eq:controllability-quasilinear} and~\eqref{eq:estimate-convolution-epsilon}, we easily get:
\begin{align}
  \frac{1}{2}\lVert w(T) \rVert^2
  & \leq \left|\iint_{\Q}(\nabla\Phi(\lvert \nabla\mathsf{v}_{\delta} \rvert) - \nabla(\Phi)^\delta(\lvert\nabla\mathsf{v}_{\delta}\rvert))\cdot \nabla(y - \mathsf{v}_{\delta})\;dx\,dt\right| + \left|\iint_{\Q}\delta\nabla\mathsf{v}_{\delta}\cdot (\nabla y - \nabla\mathsf{v}_{\delta}) \;dx\,dt \right| \nonumber\\
  &+ \frac{1}{2}\lVert M_\delta(y^0) - y^0 \rVert^2. \label{eq:bound-w(T)}
\end{align}
Evaluating the second integral term in~\eqref{eq:bound-w(T)}, we obtain:

\begin{align}
  \left|\iint_{\Q}\delta\nabla\mathsf{v}_{\delta}\cdot (\nabla y - \nabla\mathsf{v}_{\delta}) \;dx\,dt \right|& \leq \delta\lVert y - \mathsf{v}_{\delta}\rVert_{L^2((0,T),H_0^1(\Omega)}^2 \nonumber \\
                                                                            & + \delta \lVert y \rVert_{L^{\frac{p}{p-1}}\left((0,T),W_0^{1,\frac{p}{p-1}}(\Omega)\right)}\lVert y- \mathsf{v}_{\delta}\lVert_{L^p\left((0,T),W_0^{1,p}(\Omega)\right)}.
              \label{eq:epsilon-term}
\end{align}

We get, since the solutions are regular enough, that the term present in \eqref{eq:epsilon-term} goes to zero as $\delta$ does. Now, from assumption~\ref{A2}, then using \eqref{eq:defi-regul-process-R} and \eqref{def:real-line-regul-process}, we get the following estimate for the first integral term in~\eqref{eq:bound-w(T)}:
\begin{align}
  \left|\iint_{\Q}(\nabla\Phi(\lvert \nabla\mathsf{v}_{\delta} \rvert) - \nabla(\Phi)^\delta(\lvert\nabla\mathsf{v}_{\delta}\rvert))\cdot \nabla(y - \mathsf{v}_{\delta})\;dx\,dt\right| \le  &\left\|\nabla\Phi(\lvert \nabla\mathsf{v}_{\delta} \rvert) - \nabla(\Phi)^\delta(\lvert\nabla\mathsf{v}_{\delta}\rvert) \right\|_{L^{\frac{p}{p-1}}(\Q)} \nonumber \\
  &\left\|y - \mathsf{v}_{\delta} \right\|_{L^p(0, T, W_0^{1, p}(\Omega))}.\label{eq:convergence-potential}
\end{align}

Let us now formally denote $\nabla \mathsf{v} = \underset{\delta \to 0}{\underline{\mathrm{lim}}}\; \nabla\mathsf{v}_{\delta}$ (which leads to a term that can be estimated even if this limit was infinite, thanks to assumptions \ref{A1}--\ref{A3}, even if it can be proven that it is finite under suitable assumptions, see e.g. \cite[Theorem 5.2.1.]{PruessSimonett-2016} or \cite[Theorem 2.3.1. and Theorem 2.4.1.]{Zheng-2004}). Then, we can write:

\begin{equation}\label{eq:estimate-difference-regularization-potential}
\begin{array}{rl}
 \nabla\Phi(\lvert \nabla \mathsf{v}_{\delta} \rvert) - \nabla(\Phi)^{\delta}(\lvert\nabla\mathsf{v}_{\delta}\rvert) =&\; \nabla\Phi(\lvert \nabla \mathsf{v}_{\delta}\rvert) - \nabla\Phi(\lvert \nabla \mathsf{v}\rvert)\\
& + \nabla\Phi(\lvert \nabla \mathsf{v}\rvert) - (\nabla\Phi)^{\delta}(\lvert \nabla \mathsf{v}\rvert)\\
&+ (\nabla\Phi)^{\delta}(\lvert \nabla \mathsf{v}\rvert) - \nabla(\Phi)^{\delta}(\lvert \nabla \mathsf{v}\rvert)\\
&+ \nabla(\Phi)^{\delta}(\lvert \nabla \mathsf{v}\rvert) - \nabla(\Phi)^{\delta}(\lvert \nabla \mathsf{v}_{\delta}\rvert).
\end{array}
\end{equation}

Integrating the left-hand side in \eqref{eq:estimate-difference-regularization-potential}, we get by definition that the first and the fourth term in the obtained right-hand side goes to zero as $\delta$ does thanks to assumptions \ref{A1}--\ref{A3}, the second term does too using assumption \ref{A2} and since it is a classical mollification. It remains to deal with the third term, but it still goes to zero as $\delta$ does applying Friedrich's Lemma (see e.g. \cite[Section 1.5.4.]{CherrierMilani-2012}, \cite[Lemma 17.1.5]{Hoermander-2007} or \cite[Section 2.2.]{fabrie-boyer}).

Moreover, an energy estimate leads to, testing against $\mathsf{v}_{\delta}$ into the weak formulation:

\begin{equation}\label{eq:energy-estimate-v-1}
    \frac{1}{2}\lVert \mathsf{v}_{\delta} \rVert_{L^{\infty}((0,T),L^2(\Omega))}^2 + \iint_{\Q}F_{\delta}(\lvert \nabla \mathsf{v}_{\delta}\rvert)\lvert\nabla \mathsf{v}_{\delta} \rvert^2\; dx\,dt = \iint_{\q}\chi_{\omega}\varphi\nabla\mathsf{v}_{\delta}\;dx\,dt + \frac{1}{2}\lVert M_\delta(y^0) \rVert_{L^2(\Omega)}^2.
\end{equation}

We get from assumptions \ref{A1}--\ref{A2}, using Young's inequality for $0 < s < 1$ and the fact that $\lVert M_\delta(y^0) \rVert_{L^2(\Omega)} \leq \lVert y^0 \rVert_{L^2(\Omega)}$:

\begin{equation}\label{eq:energy-estimate-v-2}
\begin{array}{rl}
    \frac{1}{2}\lVert \mathsf{v}_{\delta} \rVert_{L^{\infty}((0,T),L^2(\Omega))}^2 + 2\delta\lVert \mathsf{v}_{\delta} \rVert_{L^2((0,T),H_0^1(\Omega))}^2 + &\iint_{\Q}F(\lvert \nabla \mathsf{v}_{\delta}\rvert)\lvert\nabla \mathsf{v}_{\delta} \rvert^2\; dx\,dt\\
    \\
    &\leq \left(\frac{p-1}{p^{\frac{p}{p-1}}s^{\frac{1}{p-1}}}\right)\lVert \chi_{\omega} \varphi \rVert_{L^{\frac{p}{p-1}}(\q)}^{\frac{p}{p-1}} + s \lVert \mathsf{v}_{\delta} \rVert_{L^p((0,T),W_0^{1,p}(\Omega)}^p\\
    \\
    &+ \lVert \nabla\Phi(\lvert \nabla \mathsf{v}_{\delta} \rvert)\cdot \nabla\mathsf{v}_{\delta} - \left(\nabla\Phi\right)^{\delta}(\lvert \nabla \mathsf{v}_{\delta} \rvert)\cdot \nabla\mathsf{v}_{\delta}\rVert_{L^1(\Q)}\\
    \\
    &+  \frac{1}{2}\lVert y^0 \rVert_{L^2(\Omega)}^2.
    \end{array}
\end{equation}

We point out that from assumptions \ref{A2}--\ref{A3}, the product $\nabla\Phi(\lvert \nabla \mathsf{v}_{\delta} \rvert)\cdot \nabla \mathsf{v}_{\delta}$ involved in the above inequality is non negative. Let us focus ourselves to the case $1 < p < 2$, the case $p \geq 2$ being rather direct. From assumption (A4), since we get:

\begin{equation}\label{eq:energy-estimate-v-3}
\begin{array}{rl}
    \iint_{\Q}F(\lvert \nabla \mathsf{v}_{\delta}\rvert)\lvert\nabla \mathsf{v}_{\delta} \rvert^2\; dx\,dt &\geq \iint_{\Q}(\mu + \nabla \mathsf{v}_{\delta}^2)^{\frac{p-2}{2}}\lvert \nabla \mathsf{v}_{\delta} \rvert^2\; dx\,dt\\
    \\
    &= \iint_{\Q}(\mu + \lvert \nabla \mathsf{v}_{\delta} \rvert^2)^{\frac{p}{2}}\; dx\,dt - \iint_{\Q} \mu(\mu + \lvert \nabla \mathsf{v}_{\delta}\rvert^2)^{\frac{p-2}{2}}\; dx\,dt\\
    \\
    &\geq \lVert \mathsf{v}_{\delta} \rVert_{L^p((0,T),W_0^{1,p}(\Omega))}^p - \lvert \Omega\rvert T\mu^{\frac{p}{2}}.
    \end{array},
\end{equation}

then one can write from \eqref{eq:energy-estimate-v-2} and \eqref{eq:energy-estimate-v-3} for $0 < s < \frac{1}{2}$:

\begin{equation}\label{eq:energy-estimate-v-4}
\begin{array}{rl}
     (1-s)\lVert \mathsf{v}_{\delta} \rVert_{L^p((0,T),W_0^{1,p}(\Omega))}^p
    \leq& \left(\frac{p-1}{p^{\frac{p}{p-1}}s^{\frac{1}{p-1}}}\right)\lVert \chi_{\omega} \varphi \rVert_{L^{\frac{p}{p-1}}(\q)}^{\frac{p}{p-1}}\\
    \\
    &+ \lVert \nabla\Phi(\lvert \nabla \mathsf{v}_{\delta} \rvert)\cdot \nabla\mathsf{v}_{\delta} - \left(\nabla\Phi\right)^{\delta}(\lvert \nabla \mathsf{v}_{\delta} \rvert)\cdot \nabla\mathsf{v}_{\delta}\rVert_{L^1(\Q)}\\
    \\
    &+  \frac{1}{2}\lVert y^0 \rVert_{L^2(\Omega)}^2 + \lvert \Omega\rvert T\mu^{\frac{p}{2}}.
    \end{array}
\end{equation}

Using Young's inequality, we then get from \eqref{eq:energy-estimate-v-4}:

\begin{equation}\label{eq:energy-estimate-v-5}
\begin{array}{rl}
     (1-2s)\lVert \mathsf{v}_{\delta} \rVert_{L^p((0,T),W_0^{1,p}(\Omega))}^p
    \leq& \left(\frac{p-1}{p^{\frac{p}{p-1}}s^{\frac{1}{p-1}}}\right)\lVert \chi_{\omega} \varphi \rVert_{L^{\frac{p}{p-1}}(\q)}^{\frac{p}{p-1}}\\
    \\
    &+ \left(\frac{p-1}{p^{\frac{p}{p-1}}s^{\frac{1}{p-1}}}\right)\lVert \nabla\Phi(\lvert \nabla \mathsf{v}_{\delta} \rvert) - \left(\nabla\Phi\right)^{\delta}(\lvert \nabla \mathsf{v}_{\delta} \rvert)\rVert_{L^{\frac{p}{p-1}}(\Q)}^{\frac{p}{p-1}}\\
    \\
    &+  \frac{1}{2}\lVert y^0 \rVert_{L^2(\Omega)}^2 + \lvert \Omega\rvert T\mu^{\frac{p}{2}}.
    \end{array}
\end{equation}

And so the uniform bound over $\delta$ of $\lVert \mathsf{v}_{\delta} \rVert_{L^p((0,T),W_0^{1,p}(\Omega))}$ follows since every term in the right-hand side of \eqref{eq:energy-estimate-v-5} is uniformly bounded over $\delta$ (this last being chosen small enough). Combining this fact with \eqref{eq:convergence-potential}, we get that the left-hand side of \eqref{eq:convergence-potential} goes to zero as $\delta$ does.  Thus, up to take $\delta$ small enough, from~\eqref{eq:bound-w(T)}--\eqref{eq:convergence-potential} combined to~\eqref{eq:energy-estimate-v-5}, we obtain the wished approximate controllability.
\end{proof}

\begin{remark}

Thanks to Remark~\ref{rem:double-regularization}, we can make the observation that the previous reasoning still works when considering only local regularity on the function $F$, by considering the double regularization $R_{\delta}(r_{\delta}(F)(\lvert \cdot \rvert))$. The term associated to the potential is then estimated using the following identity.

\begin{equation}\label{eq:estimate-convolution-epsilon-bis}
\begin{array}{rl}
	\nabla\Phi_{\delta}(\lvert x \rvert) &:= F_{\delta}(\lvert x \rvert)x\\
     &:= R_{\delta}\left(r_{\delta}(F)(\lvert x \vert)\right)x\\
    &=R_{\delta}\left(\zeta_{\delta}*(\sigma_{\delta}F)\lvert x \rvert + \delta\right)x\\
    &=\left(\nu_{\delta}*\chi_{\delta}\left(\zeta_{\delta}*(\sigma_{\delta}F)\lvert x \rvert + \delta\right) + \delta\right)x\\
    &= \nu_{\delta}*\chi_{\delta}\left(\zeta_{\delta}*(\sigma_{\delta}F)\lvert x \rvert + \delta\right)x + \delta x\\
    &= \nu_{\delta}*\chi_{\delta}\zeta_{\delta}*(\sigma_{\delta}F)\lvert x \rvert x + (\nu_{\delta}*\delta)x + \delta x\\
    &=\nu_{\delta}*\chi_{\delta}\zeta_{\delta}*(\sigma_{\delta}F(\lvert x \rvert) x) + 2\delta x\\
    &=\nabla (\Phi)^{\delta}(\lvert x \rvert) + 2\delta x.
    \end{array}
\end{equation}

\end{remark}

Let us now consider Theorem~\ref{thm:global-exact-control-quasi} and Corollary~\ref{coro:global-exact-control-p-laplacian}. In fact, when the solution stops in finite time, it is enough to bring its energy to be almost null so that it becomes null in an arbitrarily short time. In other words, the global approximate controllability implies the global exact controllability, as soon as the stopping time is controlled by the energy of the initial data via a relation as in \eqref{eq:est-stopping-time}.  

\begin{proof}[Proof of Theorem~\ref{thm:global-exact-control-quasi}]
 
Let $T^\star \in (0, T)$ and $\varepsilon = \left(\frac{T^\star}{2\mu}\right)^\frac{1}{\gamma}$.
Applying Theorem~\ref{thm:global-approximate-control-quasi} for an approximate control in time $\frac{T^\star}{2}$ there exists a control $\varphi_c \in L^2((0, \frac{T^*}{2}) \times \Omega)$ such that
the solution $y$ of~\eqref{eq:0} with the control given by
\[
\varphi(t) = \left\{
\begin{array}{ll}
     \varphi_c(t) & \text{ for } t \in \left(0, \frac{T^\star}{2}\right) \\
     0 &  \text{ for } t \ge \frac{T^\star}{2}
\end{array}
\right.
\]
verifies
\[
\left\|y\left(\frac{T^\star}{2}\right)\right\| \le \varepsilon.
\]
Combining the above inequality to the estimate~\eqref{eq:est-stopping-time} we obtain that $y(T^\star) = 0$, which is the desired result.
\end{proof}

The case of the parabolic $p$-Laplacian is not directly taken into account directly by Theorem~\ref{thm:global-approximate-control-quasi} (see e.g. \cite[Example 25.5.]{zeidler-nfa2b-90}) setting $\Phi(t) = \frac{1}{p}t^p$, and thus we immediately get the Corollary~\ref{coro:approx-control-p-laplacian}. However, as is customary and as we mentioned in our introduction  its solutions can be approximated by solutions of

\begin{equation}\label{eq:p-laplacian-approx}
\left\{
\begin{array}{ll}
\partial_t\mathsf{y} - \mu\Delta \mathsf{y} - {\rm div}\left(\left(\mu + \lvert \nabla \mathsf{y} \rvert^2\right)^{\frac{p-2}{2}}\nabla \mathsf{y}\right) = \chi_{\omega}\varphi& \text{ in } \Q\\
\mathsf{y} = 0& \text{ on } \Sig\\
\mathsf{y}(0) = y^0& \text{ in } \Omega.\\
\end{array}
\right.
\end{equation}

(see e.g. \cite{Lewis-1983}) which is approximately controllable according to Theorem~\ref{thm:global-approximate-control-quasi}. It is possible to see that, for example, by observing that the approximation operator in $\mu > 0$ has the so-called M-property (see \cite[Lemma 3.2.2.]{Zheng-2004}, \cite[Chapitre 2 Remarque 2.1.]{lions-quelques-resolutions}, and \cite[Proposition 31.5.]{zeidler-nfa2b-90}) and converges in the sense of $L^p((0,T),W_0^{1,p}(\Omega))$ to the $p$-Laplacian operator. As previously mentioned, Corollary~\ref{coro:global-exact-control-p-laplacian} is an immediate consequence of Theorem~\ref{thm:global-exact-control-quasi} applied to \cite[Proposition 2.1.]{dibenedetto-degenerate-parabolic} and \cite[Exemple 1.5.2.]{lions-quelques-resolutions}, setting $X := W_0^{1,p}(\Omega) \cap L^2(\Omega)$. Corollary~\ref{coro:global-exact-control-p-laplacian-delta} is also a direct consequence of \cite[Theorem 2.1.]{AntontsevDiazShmarev-2002}, setting $X:= H_0^1(\Omega) \cap H^2(\Omega)$.

\section{Numerical simulations} \label{sec:num}

The aim of this section is to propose a numerical strategy for the computation of an approximate null control for quasilinear equations~\eqref{eq:0}. In a first step we approach an approximate control $\varphi$ for the linear equation~\eqref{eq:1} by solving a mixed formulation in order to approach the solution of the optimality condition~\eqref{eq:optim}. In~\cite{munch-souza-16} the authors propose to approach an approximated control by solving the following mixed formulation: find $(\varphi, \lambda) \in \Phi \times L^2(\Q)$ solution to
\begin{equation}
    \label{eq:mf}
    \left\{
    \begin{array}{ll}
    \mathbb{a}(\varphi, \overline \varphi) + \mathbb{b}(\overline\varphi, \lambda) = \mathbb{l}(\overline \varphi)& \qquad (\overline \varphi \in \Phi) \\
    \mathbb{b}(\varphi, \overline\lambda) = 0 & \qquad (\overline \lambda \in L^2(\Q)),
    \end{array}
    \right.
\end{equation}
where
\begin{align}\label{eq:a}
 \mathbb{a} : \Phi \times \Phi \to \mathbb{R}, & \quad \mathbb{a}(\varphi, \overline \varphi) = \iint_{\q} \chi_\omega \varphi \overline \varphi \, dx\, dt + \varepsilon (\varphi(T), \overline \varphi(T)) \\
 \label{eq:b}
 \mathbb{b} : \Phi \times L^2(\Q) \to \mathbb{R}, & \quad \mathbb{b}(\varphi, \lambda) = -\iint_{\Q} (\partial_t \varphi + \textrm{div}(a \nabla \varphi)) \lambda \, dx\, dt \\
 \label{eq:l}
 \mathbb{l} : \Phi \to \mathbb{R}, & \quad \mathbb{l}(\varphi) = -(\varphi(0), y^0).
\end{align}
The space \(
\Phi
\) 
appearing in the above relations is the completion with respect to the norm
\[
\vvvert\varphi\vvvert^2 = \iint_{\q} \chi_\omega |\varphi|^2 \, dx\, dt + \varepsilon \|\varphi(T)\|^2 + \eta \|\partial_t\varphi + \textrm{div}(a \nabla \varphi)\|_{L^2(\Q)}^2
\]
of the following space:
\[
W = \left\{ 
\varphi \in C^2(\overline{\Q}),\ \varphi(T) \in C^\infty(\Omega),\, \varphi = 0 \text{ on } \Sig
\right\}.
\]
We mention that in~\cite{munch-souza-16} it was shown that the mixed formulation~\eqref{eq:mf} is wellposed, $\varphi$ is the solution of~\eqref{eq:dual} corresponding to the final data obtained as the minimum of the functional $J^a_0$ given by~\eqref{eq:reg-HUM-functional}.

In order to numerically compute an approximate control for the quasilinear equation~\eqref{eq:0} we employ the mixed formulation of the control proble combined to a fixed point strategy. This approach is illustrated by several examples in dimension one of the space. For the remaining part of this section we consider $\Omega = (0, 1)$, $\omega=(0.1, 0.5)$ and $T = 0.5$.

From a practical point of view, the proposed strategy needs to efficiently compute the solutions of mixed formulations of the form~\eqref{eq:mf}.  In order to numerically approach the solutions of such mixed formulations, we consider structured triangulations $\mathcal{T}_h$ of the domain $\Q$ with $h > 0$ being the diameter of triangles forming $\mathcal{T}_h$.
Then we define the finite dimensional sub-spaces $\Phi_h \subset \Phi$ and $\Lambda_h \subset L^2(Q_T)$ as follows:
\begin{equation}\label{eq:Phih}
    \Phi_h = \left\{
    \phi_h \in C^1(\overline{\Q})\ : \ \phi_h|_{K} \in \mathbb{P}(K) \ \forall K \in \mathcal{T}_h,\ \phi_h = 0 \text{ on } \Sig
    \right\},
\end{equation}
where $\mathbb{P}(T)$ denotes the reduced Hsieh-Clough-Tocher (HCTr for short) $C^1$ finite element space, and
\begin{equation}
    \label{eq:Lambdah}
    \Lambda_h = \left\{
    \lambda_h \in C(\overline{\Q}) \ : \ \lambda_h|_K \in \mathbb{P}_1(K) \ \forall K \in \mathcal{T}_h
    \right\},
\end{equation}
with $\mathbb{P}_1(T)$ being the space of affine functions with respect to both $x$ and $t$. We then approach the mixed formulation \eqref{eq:mf} by its following discrete version: find $(\varphi_h, \lambda_h) \in \Phi_h \times \Lambda_h$  solution to
\begin{equation}
    \label{eq:mf_h}
    \left\{
    \begin{array}{ll}
    \mathbb{a}(\varphi_h, \overline \varphi_h) + \mathbb{b}(\overline\varphi_h, \lambda_h) = \mathbb{l}(\overline \varphi_h)& \qquad (\overline \varphi_h \in \Phi_h) \\
    \mathbb{b}(\varphi_h, \overline\lambda_h) = 0 & \qquad (\overline \lambda_h \in \Lambda_h).
    \end{array}
    \right.
\end{equation}
Remark that for every $h>0$ the mixed-formulation~\eqref{eq:mf_h} is well posed. Nevertheless, in order to have a convergence of the solutions $(\varphi_h, \lambda_h)$ to the solution $(\varphi, \lambda)$ a discrete inf-sup should be verified for the discrete mixed-formulation~\eqref{eq:mf_h} with a inf-sup constant uniform with respect to~$h$. Proving such a uniform inf-sup condition is generally a difficult question. An alternative avoiding the necessity of this condition is to stabilize the mixed formulation~\eqref{eq:mf_h} by an appropriate term.

We denote by $N_x$ the number of right triangles in the triangulation $\mathcal{T}_h$ having one side on the boundary $\Omega\times \{0\}$ and by $N_y$ the number of right triangles having one side on the boundary $\{0\} \times (0, T)$. We take $N_y$ such that the vertical side $h_y$ of every triangle in $\mathcal{T}_h$ is much smaller than $h_x$ where $h_x$ is the length of the horizontal side of the triangle. Then $h_x = 1 / N_x$ and $N_y = N_x \gamma^{-1} T$ with $\gamma \in (0, 1]$ being such that $N_y$ is an integer. Two such triangulations are represented in Figure~\ref{fig:tri}.

\begin{figure}[ht!]
  \centering
  \begin{tabular}{cc}
    \includegraphics{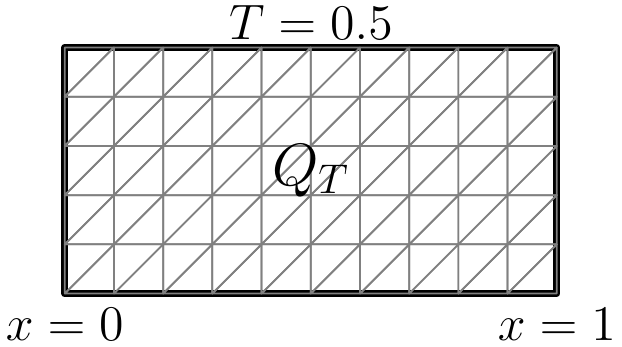} &  \includegraphics{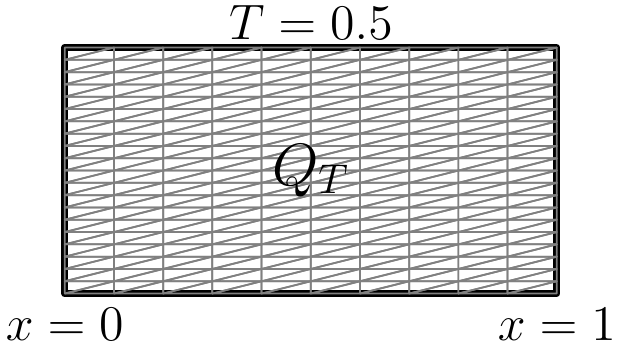}\\
    (a) & (b)
  \end{tabular}
  \caption{Two structured triangulations of $Q_T$ with $N_x = 10$. (a) $\gamma = 1$. (b) $\gamma = 0.25$.}
  \label{fig:tri}
\end{figure}

Since the controls of minimal $L^2$ norm for the heat equation oscillate in time near the control time $T$, for all the simulations discussed in this work we consider meshes that are finer in time than in space. More exactly, we take $N_y = 320$ and $N_x \in \{20,\ 40,\ 80,\ 160 \}$.

\subsection[Approximation of controls]{Approximation of controls for linear parabolic equations}
\label{ss:approx-lin}

In this section we consider a non-homogeneuous diffusion coefficient given by
\begin{equation}\label{eq:anum}
a(t, x) = \frac{1}{10} \left( 1 + x^2 + t \right).
\end{equation}
In order to compute an approximate control for the equation~\eqref{eq:1} we numerically aproach the minima of the functional $J^a_0$ by solving the mixed formulation~\eqref{eq:mf_h}. 

In what follows, we consider two examples of regular initial data to control.

\subsubsection[Example 1]{Example 1: linear equation with $\boldsymbol{y^0(x) = \sin(\pi x)}$}
\label{sss:ex1}

As a first example we consider the approximate control of the linear equation~\eqref{eq:1} with initial data given by
\begin{equation}\label{eq:id-sin}
y^0(x) = \sin(\pi x).
\end{equation}

In Table~\ref{tab:sin-L2} we gather the $L^2$ norm of the approximate control $\chi_\omega \varphi$ obtained for different meshes and three different values of $\varepsilon$. We observe that the norm of the control converges with respect to the size of the mesh for each value of $\varepsilon \in \{10^{-2i} \text{ with } 1 \le i \le 6 \}$. We observe that norm of the control are larger for smaller valuer of $\varepsilon$ and they seem to converge with respect to $N_x$ and $\varepsilon$. The control $\chi_\omega \varphi$ and its associated controlled solution $\lambda$ computed for $N_x = 160$ and $\varepsilon = 10^{-12}$ are displayed in Figure~\ref{fig:sin}.

\begin{table}[ht!]
    \centering
    \begin{tabular}{ccccccc}
    \rowcolor{gray!20} $\varepsilon$ & $10^{-2}$ & $10^{-4}$ & $10^{-6}$ &
    $10^{-8}$ & $10^{-10}$ & $10^{-12}$
    \\
\cellcolor{gray!20} $N_x = 20$ & 
0.943 &
1.946 &
2.495 &
2.690 &
2.698 &
2.698
 \\
\cellcolor{gray!20} $N_x = 40$ & 
0.930 &
1.895 &
2.422 &
2.659 &
2.678 &
2.678
 \\
\cellcolor{gray!20} $N_x = 80$ & 
0.935 &
1.905 &
2.437 &
2.690 &
2.717 &
2.718
 \\
\cellcolor{gray!20} $N_x = 160$ & 
0.936 &
1.908 &
2.442 &
2.699 &
2.730 &
2.730
    \end{tabular}
    \caption{Example 1: $L^2(\q)$ norm of the control of the linear equation~\eqref{eq:1} with a diffusion coefficient $a$ given by~\eqref{eq:anum} and initial data~\eqref{eq:id-sin} as a function of $\varepsilon$ and $N_x$.}
    \label{tab:sin-L2}
\end{table}

\begin{figure}[ht!]
    \centering
    \begin{tabular}{cc}
         \includegraphics[width=0.45\textwidth]{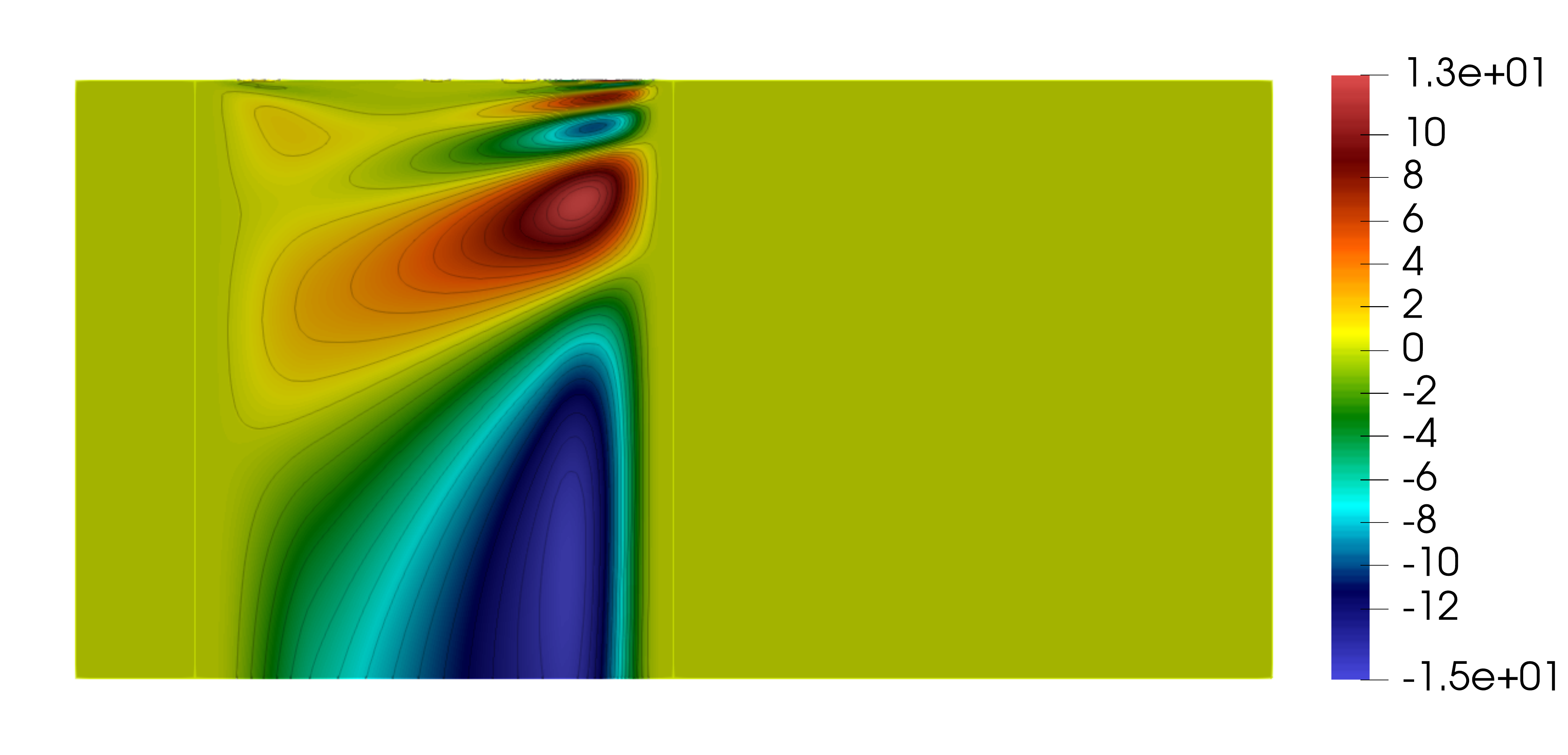}& 
         \includegraphics[width=0.45\textwidth]{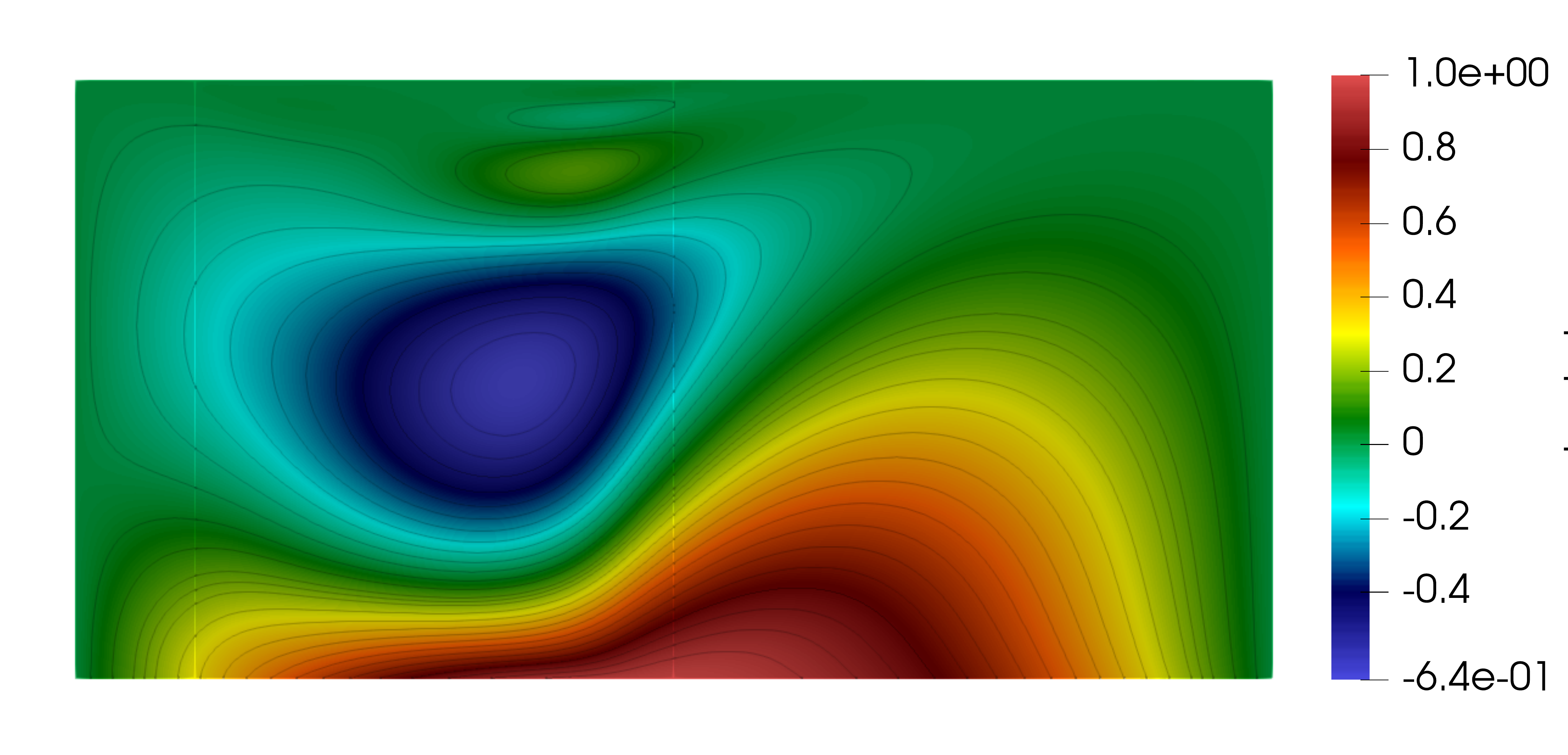} \\
         (a) & (b)
    \end{tabular}
    \caption{Example 1: (a) Control $\chi_\omega \varphi$ for the linear equation~\eqref{eq:1} with $a$ given by~\eqref{eq:anum} and initial data~\eqref{eq:id-sin} computed for $N_x=160$ and $\varepsilon=10^{-12}$. (b) The corresponding controlled solution $\lambda$.}
    \label{fig:sin}
\end{figure}

\subsubsection[Example 2]{Example 2: linear equation with $\boldsymbol{y^0(x) = \chi_{(0.6, 0.9)}(x)}$}

As a second example we consider a localized but still regular initial data to control:
\begin{equation}\label{eq:id-reg}
y^0(x) = \chi_{(0.6, 0.9)}(x) = \left\{ 
\begin{array}{ll}
     1,& \text{if }  x \in [0.6 + \delta, 0.9-\delta]\\
     0,& \text{if }  x \in (0, 1) \setminus (0.6, 0.9) \\
     e^{\alpha \left( \frac{1}{\delta^2} - \frac{1}{(x - 0.6)(0.6 + 2 \delta - x)} \right)} & \text{if } x \in (0.6, 0.6+\delta) \\
     e^{\alpha \left( \frac{1}{\delta^2} - \frac{1}{(x - 0.9 + 2\delta)(0.9 - x)} \right)} & \text{if } x \in (0.9-\delta, 0.9), \\
\end{array}
\right.
\end{equation}
with $\delta=0.1$ and $\alpha = 0.02$.

We obtain results similar to the ones in the Example 1 described in Section~\ref{sss:ex1}. The $L^2$ norm of the obtained control are listed in Table~\ref{tab:reg-L2}. We also depict the control and corresponding controlled solution computed on the mesh with $N_x=160$ and $\varepsilon = 10^{-6}$ in Figure~\ref{fig:reg}.

\begin{table}[ht!]
    \centering
    \begin{tabular}{ccccccc}
    \rowcolor{gray!20} $\varepsilon$ & $10^{-2}$ & $10^{-4}$ & $10^{-6}$ & $10^{-8}$ & $10^{-10}$ & $10^{-12}$\\
\cellcolor{gray!20} $N_x = 20$ & 
0.250 &
0.615 &
0.854 &
0.947 &
0.950 &
0.950
 \\
\cellcolor{gray!20} $N_x = 40$ & 
0.242 &
0.591 &
0.820 &
0.931 &
0.941 &
0.941
 \\
\cellcolor{gray!20} $N_x = 80$ & 
0.244 &
0.595 &
0.826 &
0.946 &
0.959 &
0.959
 \\
\cellcolor{gray!20} $N_x = 160$ & 
0.244 &
0.596 &
0.827 &
0.950 &
0.964 &
0.965
 \\
    \end{tabular}
    \caption{Example 2: $L^2(\q)$ norm of the control of the linear equation~\eqref{eq:1} with a diffusion coefficient $a$ given by~\eqref{eq:anum} and initial data~\eqref{eq:id-reg} as a function of $\varepsilon$ and $N_x$.}
    \label{tab:reg-L2}
\end{table}

\begin{figure}[ht!]
    \centering
    \begin{tabular}{cc}
         \includegraphics[width=0.45\textwidth]{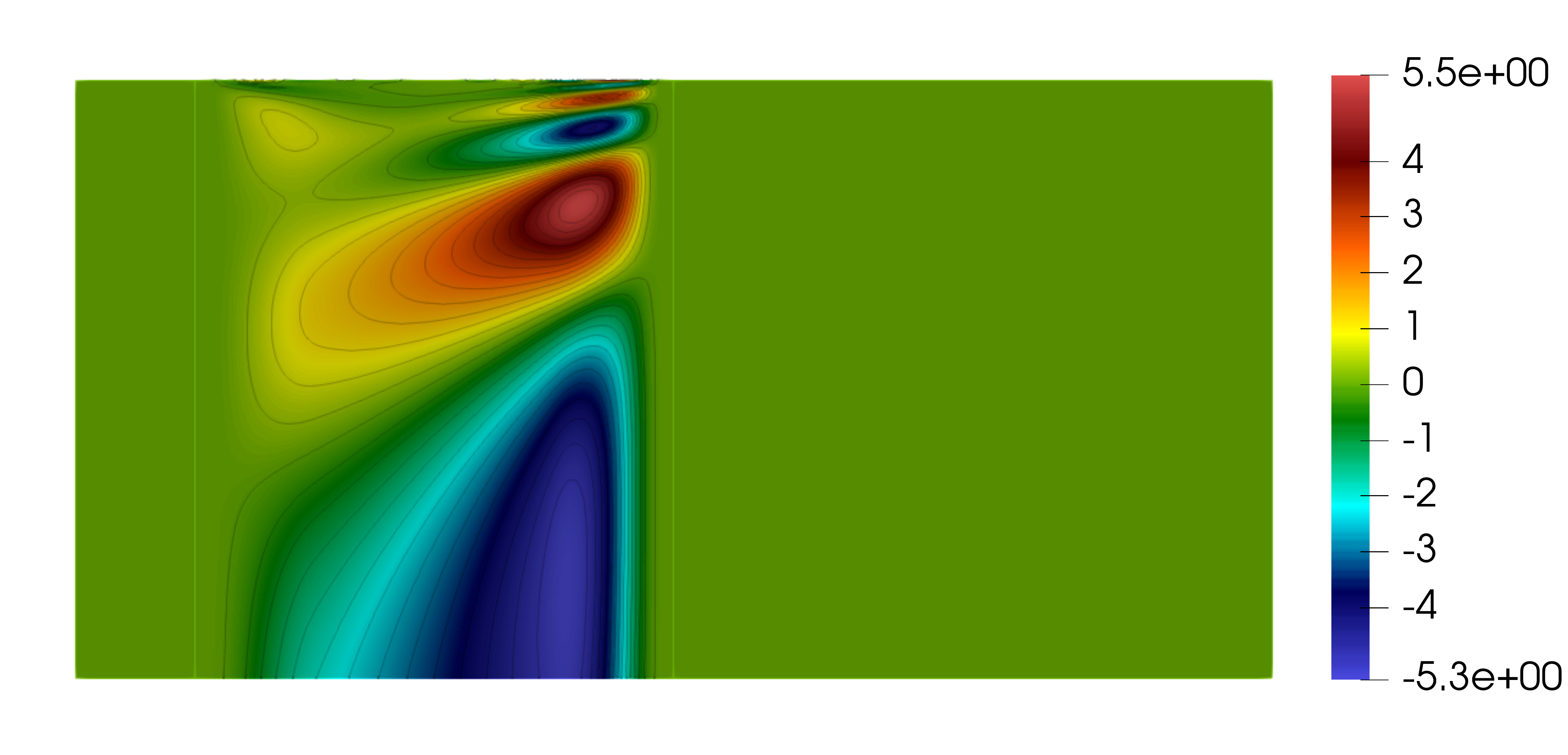}& 
         \includegraphics[width=0.45\textwidth]{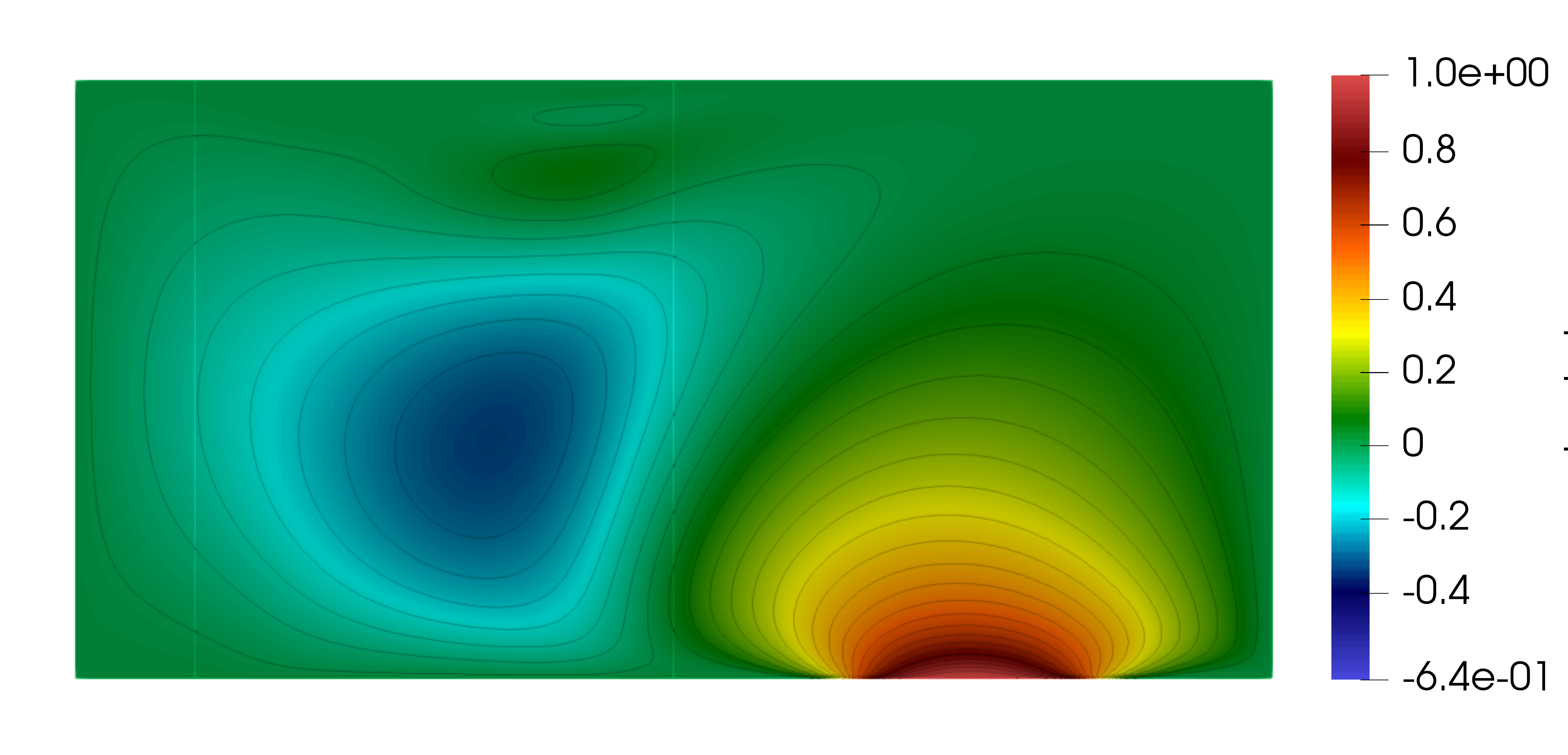} \\
         (a) & (b)
    \end{tabular}
    \caption{Example 2: (a) Control $\chi_\omega \varphi$ for the linear equation~\eqref{eq:1} with $a$ given by~\eqref{eq:anum} and initial data~\eqref{eq:id-reg} computed for $N_x=160$ and $\varepsilon=10^{-12}$. (b) The corresponding controlled solution $\lambda$.}
    \label{fig:reg}
\end{figure}

\subsection[Approximation of controls for quasilinear equations]{Approximation of controls for quasilinear equations}

For the remaining part of this section we consider the following non-linearity:
\begin{equation}\label{eq:Fnum}
F(X) = \frac{1}{10}\left(1 + (1 + X^2)^{-\frac{1}{2}}\right).
\end{equation}
Remark that this nonlinear function $F$ verifies the hypotheses~\ref{A1}--\ref{A3}.

In order to numerically approach the control and the corresponding controlled solution we employ a fixed-point algorithm combined to the strategy proposed in Section~\ref{ss:approx-lin} for the approximation of controls for linear parabolic equations. More exactly, the following algorithm is employed for the computation of an approximate null control for the quasilinear equation~\eqref{eq:0}:
\begin{algorithm}
\caption{Fixed point algorithm for the approximation of the control and the controlled solution for the quasilinear problem}\label{alg:fp}
\begin{algorithmic}
\Require $F$, $y^0$, $T$, $\varepsilon$, $i_\text{max}$ \Comment{$i_\text{max}$ is the maximal number of iterations}
\State $a \gets 1$
\State $i \gets 0$
\State $\text{err} \gets +\infty$
\State Compute the control and controlled solution $(\varphi_0, \lambda_0)$ for the linear problem.
\While{$i \le i_\text{max}$  and $\text{err} > \text{tol}$} \Comment{The tolerance tol is taken equal to $h^2$} 
\State $a \gets F(|\nabla \lambda_{i}|)$
\State $i \gets i + 1$
\State Compute the control and controlled solution $(\varphi_i, \lambda_i)$ for the linear problem.
\State $\text{err} \gets \|\chi_\omega (\varphi_i - \varphi_{i-1})\|_{L^2(\q)}$
\EndWhile
\If{$\text{err} \le \text{tol}$}
\State The algorithm converged.
\State The control and solution of the quasiliinear problem are $(\varphi, \lambda) \gets (\varphi_i, \lambda_i)$.
\EndIf
\end{algorithmic}
\end{algorithm}

In what follows we consider the same initial data as in Section~\ref{ss:approx-lin} for the control of the quasilinear equation~\eqref{eq:0} corresponding to this choice of $F$. We consider different levels of meshes and several values of the penalization parameter $\varepsilon$. For each mesh of the domain $\Q$ and every value of $\varepsilon$ we compute the $L^2$ norm of the control provided by Algorithm~\ref{alg:fp} and we report the number of iterations needed for its convergence.

\subsubsection[Example 3]{Example 3: quasilinear equation with $\boldsymbol{y^0(x) = \sin(\pi x)}$}
In this section we consider again the control of initial data~\eqref{eq:id-sin} in the case of the quasilinear equation~\eqref{eq:0} with $F$ given by~\eqref{eq:Fnum}. The first question we would want to investigate is related to the convergence of Algorithm~\ref{alg:fp}. In this purpose we list in Table~\ref{tab:nl-sin-it} the number of iterations needed for the convergence of the fixed point algorithm for four levels of meshes and for four different values of the penalization parameter~$\varepsilon$. We observe that, for every $\varepsilon \in \{10^{-2i} \text{ with } 1 \le i \le 6 \}$ the number of iterations needed for the convergence slightly increases with $N_x$. This is probably due to the fact that tolerance parameter in the algorithm is smaller for larger values of $N_x$. The second observation is that the fixed point algorithm does not converge for small values of $\varepsilon$ and fine enough meshes. 
\begin{table}[ht!]
  \centering
  \begin{tabular}{ccccccc}
    \rowcolor{gray!20}
    $\varepsilon$
    & $10^{-2}$
    & $10^{-4}$
    & $10^{-6}$
    & $10^{-8}$
    & $10^{-10}$
    & $10^{-12}$
    \\
\cellcolor{gray!20} $N_x = 20$ & 
4 &
5 &
6 &
6 &
6 &
7
 \\
\cellcolor{gray!20} $N_x = 40$ & 
4 &
6 &
7 &
7 &
8 &
8
 \\
\cellcolor{gray!20} $N_x = 80$ & 
5 &
7 &
8 &
9 &
9 &
-
 \\
\cellcolor{gray!20} $N_x = 160$ & 
5 &
8 &
9 &
- &
- &
-
\\
  \end{tabular}
  \caption{Example 3: The number of iterations needed for the convergence of Algorithm~\ref{alg:fp} as a function of $\varepsilon$ and $N_x$ for the control of quasilinear equation~\eqref{eq:0} whith $F$ given by~\eqref{eq:Fnum} and initial data~\eqref{eq:id-sin}.}
  \label{tab:nl-sin-it}
\end{table}

In Table~\ref{tab:nl-sin-L2} we gather the norm of the approximate control computed for different of values of $N_x$ and $\varepsilon$. As for the control of the linear equation we observe a convergence of the norm of the control with respect to $h$. The control obtained for $N_x = 160$ and $\varepsilon = 10^{-6}$ and its associated controlled solution are illustrated in Figure~\ref{fig:nl-sin}.
\begin{table}[ht!]
  \centering
  \begin{tabular}{ccccccc}
\rowcolor{gray!20}
    $\varepsilon$
    & $10^{-2}$
    & $10^{-4}$
    & $10^{-6}$
    & $10^{-8}$
    & $10^{-10}$
    & $10^{-12}$
    \\
\cellcolor{gray!20} $N_x = 20$ & 
0.485 &
1.423 &
2.366 &
3.011 &
3.356 &
3.389
 \\
\cellcolor{gray!20} $N_x = 40$ & 
0.486 &
1.390 &
2.301 &
2.931 &
3.317 &
3.402
 \\
\cellcolor{gray!20} $N_x = 80$ & 
0.488 &
1.395 &
2.316 &
2.956 &
3.359 &
-
 \\
\cellcolor{gray!20} $N_x = 160$ & 
0.489 &
1.396 &
2.319 &
- &
- &
-
 \\    
  \end{tabular}
  \caption{Example 3: $L^2(\q)$ norm of the control of the quasilinear equation~\eqref{eq:0} with~$F$ given by~\eqref{eq:Fnum} and initial data~\eqref{eq:id-sin} as a function of $\varepsilon$ and $N_x$.}
  \label{tab:nl-sin-L2}
\end{table}

\begin{figure}[ht!]
    \centering
    \begin{tabular}{cc}
         \includegraphics[width=0.45\textwidth]{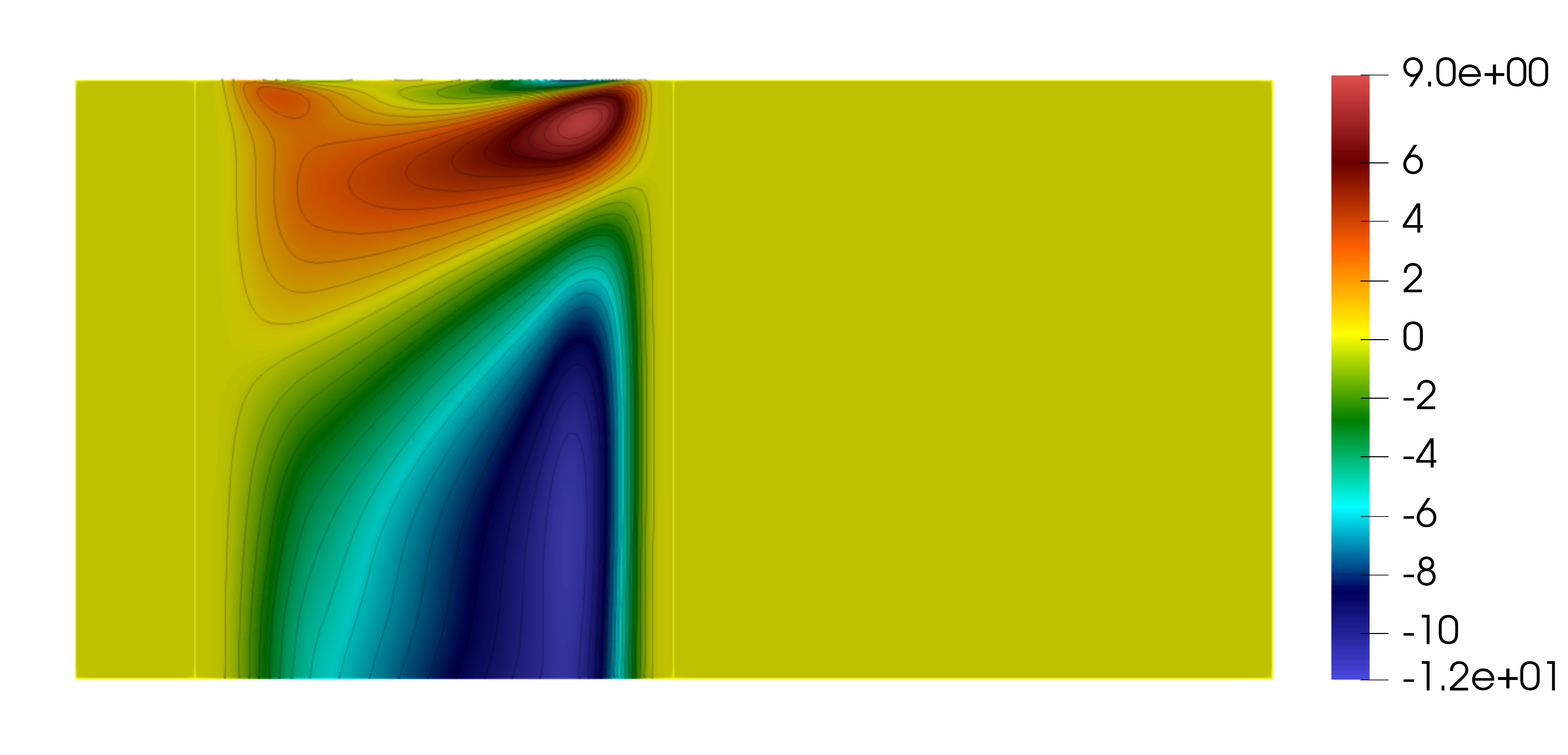}& 
         \includegraphics[width=0.45\textwidth]{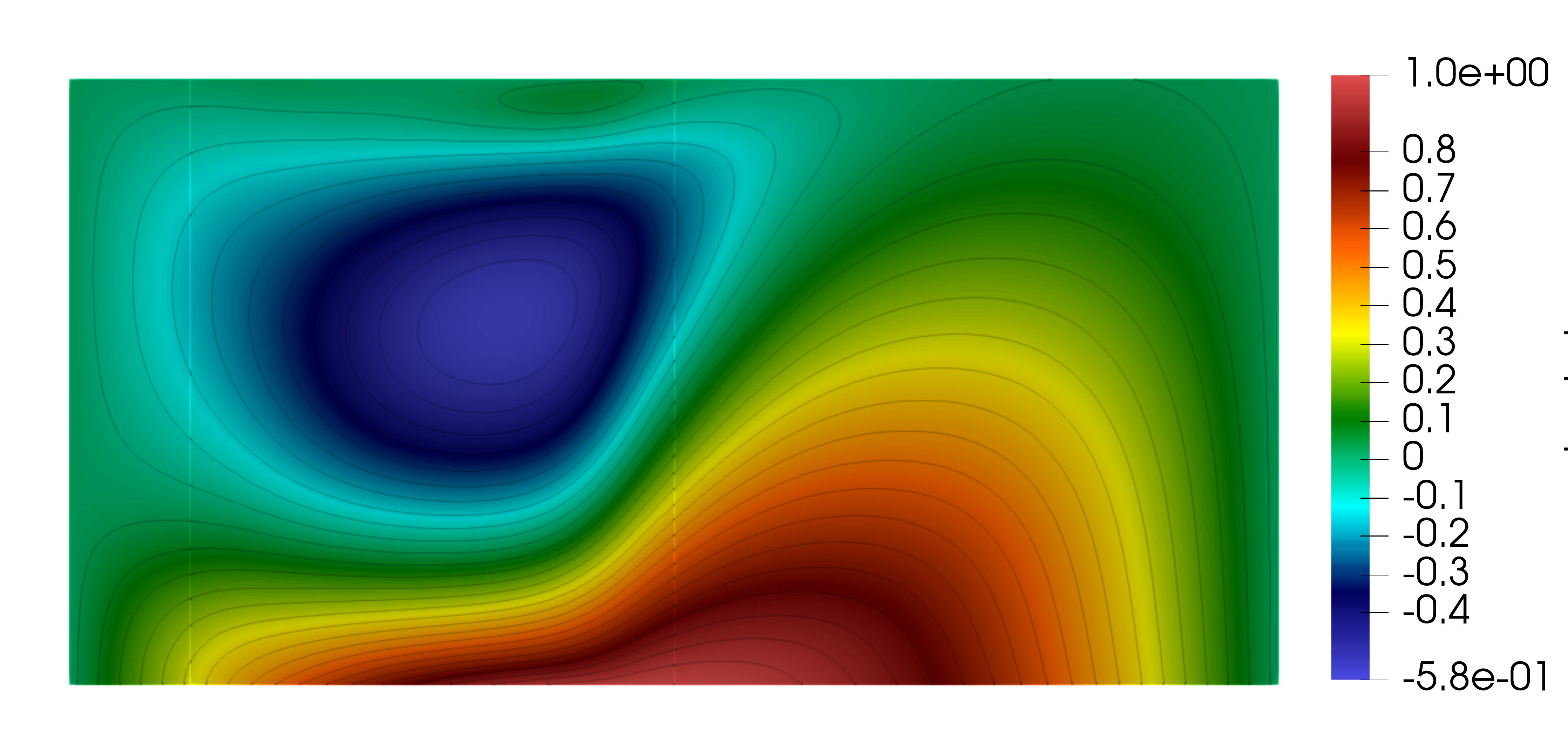} \\
         (a) & (b)
    \end{tabular}
    \caption{Example 3: (a) Control $\chi_\omega \varphi$ of the quasilinear equation~\eqref{eq:0} with $F$ given by~\eqref{eq:Fnum}, initial data given by~\eqref{eq:id-sin} and for $N_x=160$, $\varepsilon=10^{-6}$. (b) The corresponding controlled solution $\lambda$.}
    \label{fig:nl-sin}
\end{figure}

\subsubsection[Example 4]{Example 4: quasilinear equation with $\boldsymbol{y^0(x) = \chi_{(0.6, 0.9)}(x)}$}

This last example consider the numerical approximation of the approximate null control for equation~\eqref{eq:0} with $F$ given by~\eqref{eq:Fnum} and initial data~\eqref{eq:id-reg}. For this choice of initial data we conduct the same experiments as for Example~3. We obtain similar results with the difference that Algorithm~\ref{alg:fp} has a better convergence for this initial data. As reported in Table~\ref{tab:nl-reg-it} the fixed point algorithm converge for $\varepsilon=10^{-12}$ and for all the values of the discretization parameter $N_x$. Nevertheless, the number of iterations augment for $N_x = 160$ and the convergence will probably deteriorate for smaller values of $\varepsilon$.

\begin{table}[ht!]
    \centering
    \begin{tabular}{ccccccc}
\rowcolor{gray!20}
    $\varepsilon$
    & $10^{-2}$
    & $10^{-4}$
    & $10^{-6}$
    & $10^{-8}$
    & $10^{-10}$
    & $10^{-12}$
    \\
\cellcolor{gray!20} $N_x = 20$ & 
4 &
5 &
6 &
6 &
6 &
6
 \\
\cellcolor{gray!20} $N_x = 40$ & 
4 &
5 &
6 &
7 &
7 &
7
 \\
\cellcolor{gray!20} $N_x = 80$ & 
5 &
6 &
7 &
8 &
8 &
8
 \\
\cellcolor{gray!20} $N_x = 160$ & 
6 &
7 &
8 &
9 &
9 &
-
 \\
    \end{tabular}
    \caption{Example 4: the number of iterations needed for the convergence of Algorithm~\ref{alg:fp} as a function of $\varepsilon$ and $N_x$ for the control of quasilinear equation~\eqref{eq:0} whith $F$ given by~\eqref{eq:Fnum} and initial data~\eqref{eq:id-reg}.}
    \label{tab:nl-reg-it}
\end{table}

The values of the $L^2$ norm of the computed controls, reported in Table~\ref{tab:nl-reg-L2}, indicate that controls converge with respect to $N_x$ for fixed values of $\varepsilon$. This convergence seems faster for larger value of the penalization parameter $\varepsilon$. The control and the corresponding controlled solution associated to the initial data~\eqref{eq:id-reg} are displayed in Figure~\ref{fig:nl-reg}.
\begin{table}[ht!]
    \centering
    \begin{tabular}{ccccccc}
\rowcolor{gray!20}
    $\varepsilon$
    & $10^{-2}$
    & $10^{-4}$
    & $10^{-6}$
    & $10^{-8}$
    & $10^{-10}$
    & $10^{-12}$
    \\
\cellcolor{gray!20} $N_x = 20$ & 
0.127 &
0.419 &
0.748 &
1.013 &
1.166 &
1.181
 \\
\cellcolor{gray!20} $N_x = 40$ & 
0.124 &
0.405 &
0.720 &
0.976 &
1.147 &
1.185
 \\
\cellcolor{gray!20} $N_x = 80$ & 
0.124 &
0.406 &
0.723 &
0.983 &
1.161 &
1.213
 \\
\cellcolor{gray!20} $N_x = 160$ & 
0.124 &
0.406 &
0.723 &
0.984 &
1.163 &
-
 \\
    \end{tabular}
    \caption{Example 4: $L^2(\q)$ norm of the control of the quasilinear equation~\eqref{eq:0} with~$F$ given by~\eqref{eq:Fnum} and initial data~\eqref{eq:id-reg} as a function of $\varepsilon$ and $N_x$.}
    \label{tab:nl-reg-L2}
  \end{table}

\begin{figure}[ht!]
    \centering
    \begin{tabular}{cc}
         \includegraphics[width=0.45\textwidth]{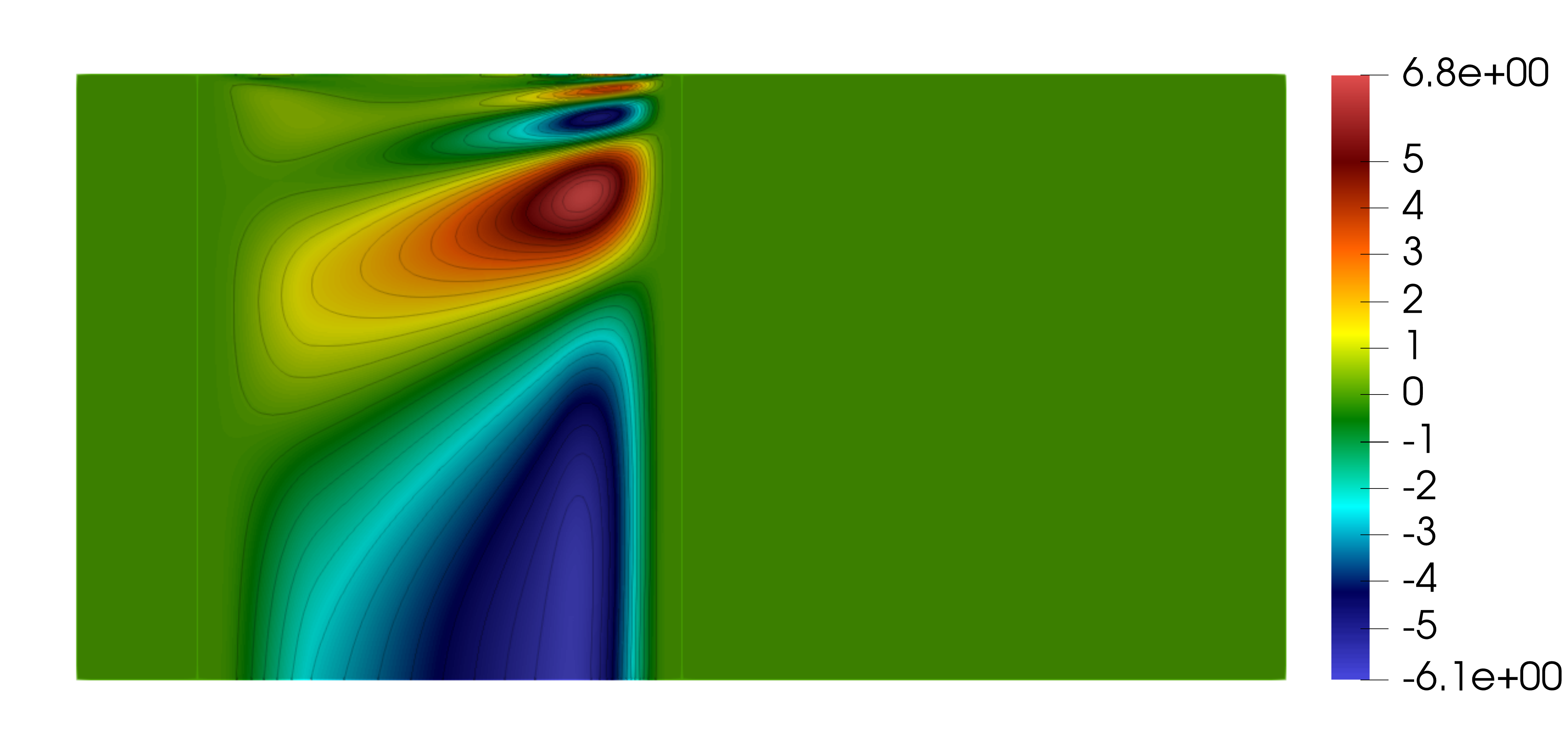}& 
         \includegraphics[width=0.45\textwidth]{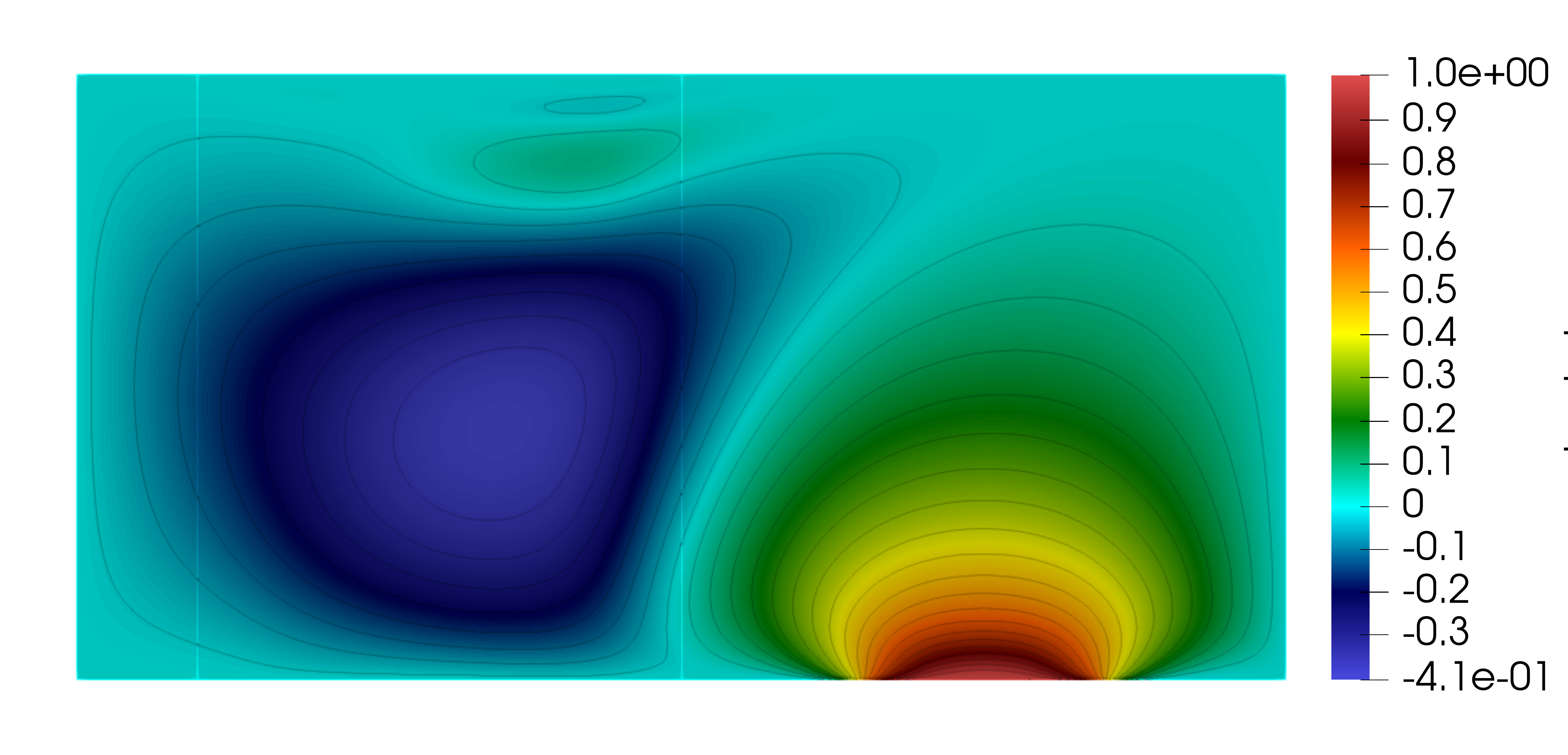} \\
         (a) & (b)
    \end{tabular}
    \caption{Example 4. (a) Control $\chi_\omega \varphi$ of the quasilinear equation~\eqref{eq:0} with $F$ given by~\eqref{eq:Fnum}, initial data given by~\eqref{eq:id-reg} and for $N_x=160$, $\varepsilon=10^{-10}$. (b) The corresponding controlled solution $\lambda$.}
    \label{fig:nl-reg}
\end{figure}

\FloatBarrier
\section{Conclusion and perspectives}

In this paper, we proved the approximate null controllability in arbitrarily small time of quasilinear equations with a gradient dependent viscosity coefficients. This class of equation includes the parabolic $p$-Laplacian equation with $\frac{3}{2} < p < 3$. Moreover, for equations, such the parabolic $p$-Laplacian with $\frac{3}{2} < p < 2$, having a finite stopping time without control, we prove the null controllability in arbitrary small time. Numerical simulations illustrate the proposed control strategy. 

A first open question is that, under the hypothesis of showing the Lipschitz continuity of the control associated to the linear problem in the $L^q$ frame for some $q > 2$, it is possible to extend our controllability result for the $L^q$ controllability of the parabolic $p$-Laplacian. We could then obtain the  exact controllability of the parabolic $p$-Laplacian for some $p_{\star} \leq p < 2$, where $1 < p_{\star} < \frac{3}{2}$, still applying \cite[Proposition 2.1.]{dibenedetto-degenerate-parabolic}. 

Another interesting question is that the results presented in this paper could be extended to the controllability of non-Newtonian fluid flows, e.g. of power law or Carreau-Yasuda type. More precisely, the issue is that for a system with solutions being divergence free in the weak $L^2$ sense, the addition of the nonlinear quadratic term will probably cause some difficulties. Our results adapt, under a few additional assumptions, to the controllability framework in the case of a system (i.e. in the non-scalar case), but it may then be necessary to regularize further in order to obtain satisfactory regularity properties (see for example \cite{BerselliRuuzicka-2022} or \cite{CianchiMazya-2020} for recent results in this framework).

\bibliographystyle{plain}
\bibliography{Bibliographie}

\end{document}